\documentclass[a4paper]{article}
\usepackage{geometry}
\usepackage[utf8]{inputenc}
\usepackage{graphicx}
\usepackage{amsmath}
\usepackage{amsthm}
\usepackage{amssymb}
\usepackage{float}
\usepackage{caption}
\usepackage{hyperref}
\usepackage{algorithm}
\usepackage{xcolor}
\usepackage{mathtools}
\usepackage{dsfont}
\usepackage{tikz}
\usepackage{comment}
\usepackage[algo2e,ruled,vlined]{algorithm2e} 
\usepackage{algorithm}
\usepackage{algorithmicx}
\usepackage{stmaryrd}

\usepackage[noend]{algpseudocode}

\usepackage[backend=bibtex, style=numeric, maxbibnames=5]{biblatex}
\mathtoolsset{showonlyrefs}

\DeclareMathOperator{\R}{\mathbb{R}}

\def \E{\mathbb{E}}
\def \F{\mathbb{F}}

\def \K{\mathbb{K}}

\def \P{\mathbb{P}}

\def\1{{\bf 1}}

\def\Fc{{\cal F}}

\def\Uc{{\cal U}}

\def\Kc{{\cal K}}
\def\Ac{{\cal A}}

\def\Pc{{\cal P}}
\def\Wc{{\cal W}}
\def\Yc{{\cal Y}}

\def\Zc{{\cal Z}}

\def\argmin_#1{\underset{#1}{\mathrm{argmin\, }}}
\newtheorem{Theorem}{Theorem}[section]
\newtheorem{rem}[Theorem]{Remark}
\newtheorem{lem}[Theorem]{Lemma}
\newtheorem{assumption}[Theorem]{Assumption}
\newtheorem{Pro}[Theorem]{Proposition}

\newcommand{\di}{\mathrm{d}}

\numberwithin{equation}{section} 

\def \mfa{\mathfrak{a}}

\def \mrw{\mathrm{w}}

\begin{document}
\title{A level-set approach to the control of state-constrained McKean-Vlasov equations: application to renewable energy storage and portfolio selection\thanks{This work  was supported by FiME (Finance for Energy Market Research Centre) and
the ``Finance et D\'eveloppement Durable - Approches Quantitatives'' EDF - CACIB Chair.
We thank Marianne Akian,  Olivier Bokanowski, and Nadia Oudjane for useful comments. 
}}

\author{Maximilien \sc{Germain}\footnote{EDF R\&D, Université de Paris and Sorbonne Université, CNRS, Laboratoire de Probabilités, Statistique et Modélisation, F-75013 Paris, France  \sf \href{mailto:maximilien.germain at edf.fr}{mgermain at lpsm.paris}} \and
Huy\^en \sc{Pham}\footnote{Université de Paris and Sorbonne Université, CNRS, Laboratoire de Probabilités, Statistique et Modélisation, F-75013 Paris, France,  CREST-ENSAE \& FiME \sf \href{mailto:pham at lpsm.paris}{pham at lpsm.paris}} 
\and
Xavier \sc{Warin}\footnote{EDF R\&D \& FiME \sf \href{mailto:xavier.warin at edf.fr}{xavier.warin at edf.fr}}} 


\date{{\it to appear in Numerical Algebra, Control and Optimization}}

\maketitle

\begin{abstract}
We consider the control of McKean-Vlasov dynamics (or mean-field control) with probabilistic state constraints. We rely on a level-set approach which provides a representation of  the constrained problem in terms of  an unconstrained one with exact penalization and running maximum or integral cost. 
The method  is then  extended to the common noise setting. Our work extends (Bokanowski, Picarelli, and Zidani, SIAM J. Control Optim. 54.5 (2016), pp. 2568–2593) and  (Bokanowski, Picarelli, and Zidani, Appl. Math. Optim. 71 (2015), pp. 125–163) to a mean-field setting. 

The reformulation as an unconstrained problem is particularly suitable for the numerical resolution of the problem, that is achieved from  an extension of a machine learning algorithm from (Carmona, Laurière, arXiv:1908.01613 to appear in Ann. Appl. Prob., 2019). A first application concerns the storage of renewable electricity in the presence of mean-field price impact and another one focuses on a mean-variance portfolio selection problem with probabilistic constraints on the wealth. 
We also illustrate our approach for a direct numerical resolution of the primal Markowitz continuous-time problem without relying on duality. 
\end{abstract}%


\vspace{3mm}




\section{Introduction}

The control of McKean-Vlasov dynamics, also known as mean-field control problem,  
has attracted a lot of interest over the last years since the emergence of the mean-field game theory. There is now an important literature on this topic addressing on one hand the theoretical aspects either by dynamic programming approach (see 
\cite{LP14,PW17,PW18,CP19}), or by maximum principle (see \cite{CD43}), 
and on the other hand the numerous applications in economics and finance, and  
we refer to the two-volume monographs \cite{cardel19,cardel19b} for an exhaustive and detailed  treatment of this area.

In this paper, we aim to study control of McKean-Vlasov dynamics under the additional presence of state constraints in law. The consideration of probabilistic constraints (usually in expectation or in target form) for standard stochastic control has many practical applications, notably in finance with quantile and CVaR type constraints,  and is the subject of many papers, we refer to  \cite{ST02,BEI10,surveyConstraints,chowetal,PTZ21,BALATA2021640} for an overview.

There exists some recent works dealing with mean-field control under some specific law state constraints. For example, the paper \cite{CW19} solves mean-field control with delay and smooth expectation terminal constraint (and without dependence with respect to the law of the control).  In the case of mean field games, state constraints are considered by \cite{CC18,CCC18,FuHorst,JGM21,AM21}. In these cited works the state belongs to a compact set, which corresponds to a particular case of our  constraints in distribution. Related literature includes the recent work  
\cite{BDK20} which studies a mean-field target problem where the aim is to find the initial laws of a controlled McKean-Vlasov process satisfying a law constraint, but only at terminal time. The paper \cite{D20} also studies these terminal constraint in law for the control of a standard diffusion process. Next, it has been extended in \cite{D21} to a running law constraint for the control of a standard diffusion process with McKean-Vlasov type cost through the control of a Fokker-Planck equation. Several works also 
consider directly the optimal control of Fokker-Planck equations in the Wasserstein space with terminal or running constraints, such as \cite{B19,BF21} through Pontryagin principle, in the deterministic case without diffusion. 

In this paper, we consider general running (at discrete or continuous time) 
and terminal constraints in law, and extend  the level-set approach \cite{BPZ15,BPZ16} (see also \cite{ABZ13} in the deterministic case) to our mean-field setting. This enables us to reformulate the constrained McKean-Vlasov control problem into an unconstrained mean-field control problem with an auxiliary state variable, and a running path-dependent supremum cost or alternatively a non path-dependent integral cost over the constrained functions. 
Such equivalent representations of the control problem with exact penalization turns out to be quite useful  for an efficient numerical resolution of the original constrained mean-field control problem. We shall actually adapt the machine learning algorithm in \cite{CL19} for solving two applications in renewable energy storage and in portfolio selection. 

The outline of the paper is organized as follows. Section \ref{sec: no common noise} develops the level-set approach in our constrained mean-field setting with supremum term.  
We present in Section \ref{sec: integral} the alternative level-set formulation with integral term, and discuss when the optimization over open-loop controls yields the same value than the optimization over closed-loop controls. This will be useful for numerical purpose in the approximation of optimal controls.  
The method is then extended in Section \ref{sec:noise} to the common noise setting.
Finally, we present in Section \ref{sec:numerical} the applications and numerical tests.

\section{Mean-field control with state constraints}\label{sec: no common noise}

Let $(\Omega,\Fc,\P)$ be a probability space on which is defined a $d$-dimensional Brownian motion $W$ with associated filtration $\F=(\Fc_t)_t$ augmented with $\P$-null sets. We assume that $\Fc_0$ is ``rich enough" in the sense that any   probability measure $\mu$ on $\R^d$ can be represented as the distribution law of some $\Fc_0$-measurable random variable. This is satisfied whenever the probability space 
$(\Omega,\Fc_0,\P)$ is atomless,  see \cite{cardel19},  p.352.

We consider the following cost and dynamics:
\begin{align}\label{eq: original problem}
J(X_0,\alpha) & = \E\Big[\int_0^T f\big(s,X_s^{\alpha},\alpha_s, \P_{(X_s^{\alpha},\alpha_s)}\big)\ \di s + g\big(X_T^{\alpha}, \P_{X_T^{\alpha}}\big) \Big]  \\
    X_t^{\alpha} & = X_0 + \int_0^t b\big(s,X_s^{\alpha},\alpha_s,\P_{(X_s^{\alpha},\alpha_s)})\ \di s + \int_0^t \sigma(s,X_s^{\alpha},\alpha_t,\P_{(X_s^{\alpha},\alpha_s)}\big)\ \di W_s,  \label{eq: MKV SDE}
\end{align} where $\P_{(X_s^{\alpha},\alpha_s)}$ 
is the joint law of $(X_s^{\alpha},\alpha_s)$ under $\P$ and $X_0 $ is a given random variable in $L^2(\Fc_0,\R^d)$. The control $\alpha$ belongs to a set $\Ac$ of $\F$-progressively measurable processes with values in a  set $A\subseteq\R^q$.  
The coefficients $b$ and $\sigma$ are  measurable functions from $[0,T]\times\R^d\times A\times\Pc_2(\R^d\times A)$ into $\R^d$ and $\R^{d\times d}$,  where $\Pc_2(E)$ is the  set of square integrable probability measures on the metric space $E$, equipped with the $2$-Wasserstein distance $\Wc_2$. 
We make some standard Lipschitz conditions on $b,\sigma$ in order to ensure that equation \eqref{eq: MKV SDE} is well-defined and admits a unique strong solution, which is square-integrable. The function $f$ is a real-valued measurable function on  $[0,T]\times\R^d\times A\times\Pc_2(\R^d\times A)$, while 
$g$ is a measurable function on $\R^d\times\Pc_2(\R^d)$, and we assume that $f$ and $g$ satisfy some linear growth condition which ensures that the functional in \eqref{eq: original problem} is well-defined.

Furthermore, the law of the controlled McKean-Vlasov process $X$ is constrained to verify 
\begin{align} \label{consPSI}  
\Psi( t,\P_{X_t^{\alpha}}) &  \leq \;  0, \quad 0 \leq t \leq T,
\end{align} 
here  $\Psi_l$, $l$ $=$ $1,\ldots,k$, are  given functions from $[0,T] \times \Pc_2(\R^d)$ into $\R$,  and  $\Psi$ $:=$ $\max_{1\leq l\leq k}\Psi_l$.
In other words, the  constraint $\Psi(t,\mu)$ $\leq$ $0$ means that  $\Psi^l(t,\mu)$ $\leq$ $0$, $l$ $=$ $1,\cdots,k$.  The problem of interest is therefore 
\begin{align}
    V & : = \;  \inf_{\alpha\in\Ac} 
    \big\{J(X_0,\alpha)  :  \Psi(t,\P_{X_t^{\alpha}})\leq 0 ,\ \forall\ t\in[0,T]  \big\}. 
\end{align}
By convention the infimum of the empty set is $+\infty$. When needed, we will sometimes use the notation $ V^\Psi$ to emphasize the dependence of the value function on $\Psi$. 
Clearly, 
$\Psi \leq \Psi' $ implies $ V^{\Psi} \leq V^{\Psi'}$.

\begin{rem}\label{rem: constraints}
This very general type of constraints includes for instance: \begin{itemize}
    \item Controlled McKean-Vlasov process $X$ constrained to stay inside a non-empty closed set $\Kc_t\subseteq\R^d$ with probability larger than a threshold $p_t\in[0,1]$, namely 
\begin{align}
    \P(X_t^{\alpha} \in \Kc_t) \geq p_t ,\ \forall\ t\in[0,T],
\end{align} with $\Psi:(t,\mu) \mapsto p_t - \mu(\Kc_t)$. With $p_t = 1,\ \forall t\in[0,T]$ it yields almost sure constraints.
\item Almost sure contraints on the state, $X_t^\alpha\in\Kc_t,\ \forall t\in[0,T]\ \P$ a.s., with \begin{equation}
    \Psi:(t,\mu) \mapsto \int_{\R^d} d_{\Kc_t}(x)\ \mu(\di x),
\end{equation} where $d_{\Kc_t}$ is the distance function to the non-empty closed set $\Kc_t$. 
\item The case of a Wasserstein ball constraint around a benchmark law $\eta_t$ in the form $\Wc_2(\P_{X_t^{\alpha}},\eta_t) \leq \delta_t$ with
\begin{equation}
    \Psi:(t,\mu) \mapsto \Wc_2(\mu,\eta_t) - \delta_t.
\end{equation} This is the constraint considered in \cite{PJ21} at terminal time.
\item A terminal constraint in law $\varphi(\P_{X_T^{\alpha}})\leq 0$ as in \cite{D20} with \begin{equation}\Psi:(t,\mu) \mapsto
   \varphi(\mu)\mathds{1}_{t=T}.
\end{equation} 
\item Terminal constraint in law $\P_{X_T^{\alpha}}\in\K\subset\Pc_2(\R^d)$ as in \cite{BDK20} with \begin{equation}\Psi:(t,\mu) \mapsto
   (1 - \mathds{1}_{\mu\in\K })\mathds{1}_{t=T}.
\end{equation}
\item The case of discrete time constraints $\phi(t_i,\P_{X_{t_i}^{\alpha}})\leq 0$ for $t_1<\cdots<t_k$ with \begin{equation}\Psi:(t,\mu) \mapsto
   \phi(t,\mu)\mathds{1}_{t\in \{t_1,\cdots,t_k \}}.
\end{equation}
\end{itemize}
\end{rem}

Even though this problem seems much more involved than the standard stochastic control problem with state constraints investigated in \cite{BPZ16}, thanks to an adequate reformulation, it turns out that we can adapt the main ideas from this paper to our framework and construct similarly an unconstrained auxiliary problem (in infinite dimension).

\subsection{A target problem and an associated control problem}

Given $z$ $\in$ $\R$, and $\alpha$ $\in$ $\Ac$,  define a new state variable
\begin{align}\label{eq: definition Z}
    Z^{z,\alpha}_{t} := z - \E\Big[\int_0^t f\big(s,X_s^{\alpha},\alpha_s, \P_{(X_s^{\alpha},\alpha_s)}\big)\ \di s\Big] = z - \int_0^t \widehat{f}\big(s,\P_{(X_s^{\alpha},\alpha_s)}\big)\ \di s, \quad 0 \leq t \leq T, 
\end{align}
where  $\widehat f$ is the function defined on  $[0,T]\times\Pc_2(\R^d\times A)$ by  $\widehat f(t,\nu)$ $=$  $\int_{\R^d\times A} f\big(t,x,a, \nu\big)\ \nu(\di x, \di a)$.  We also denote by $\widehat g$ the function defined on $\Pc_2(\R^d)$ by 
$\widehat g(\mu)$ $=$ $\int_{\R^d} g(x,\mu) \mu(\di x)$. 
\begin{lem} \label{lem: target}
The value function admits the \textbf{deterministic target problem} representation
\begin{align}\label{eq: target problem}
    V = \inf \{z\in\R\ |\ \exists\ \alpha\in\Ac\ \mathrm{s.t.}\ \widehat{g}(\P_{X_T^{\alpha}})\leq Z^{z,\alpha}_T,\ \Psi(s,\P_{X_s^{\alpha}})\leq 0,\ \forall\ s\in[0,T]   \}.
\end{align}
\end{lem}

\begin{proof}
We first observe from the definition of $V$ in \eqref{consPSI} that it can be rewritten as 
\begin{align}
    V = \inf \{z\in\R\ |\ \exists\ \alpha\in\Ac\ \mathrm{s.t.}\ J(X_0,\alpha)\leq z,\ \Psi(t,\P_{X_t^{\alpha}})\leq 0,\ \forall\ t\in[0,T] \}.
\end{align} 
Next, by noting that the cost functional is written as 
\begin{align*}
J(X_0,\alpha) &= \; \int_0^T \widehat{f}\big(t,\P_{(X_t^{\alpha},\alpha_t)}\big)\ \di t +  \widehat{g}(\P_{X_T^{\alpha}}),   
\end{align*}
the result then follows immediately by the definition of $Z^{z,\alpha}$ in \eqref{eq: definition Z}. 
\end{proof}

We want to link this representation to the \textbf{zero-level set} of the solution of an auxiliary unconstrained control problem. Define the \textbf{auxiliary unconstrained deterministic} control problem: 
 
\begin{align}\label{eq: auxiliary problem}
 \Yc^\Psi: z\in\R \mapsto \inf_{\alpha\in\Ac} \Big[\{\widehat{g}(\P_{X_T^{\alpha}})- Z^{z,\alpha}_{T} \}_+ +  \sup_{s\in[0,T]} \{\Psi(s,\P_{X_s^{\alpha}})\}_+ \Big], 
\end{align}
with the notation $\{x\}_+ = \max(x,0)$ for the positive part. We see that $\Yc^\Psi(z)\geq 0$. 

By classical estimates on McKean-Vlasov equations we can obtain continuity and growth conditions on $\Yc^\Psi$. 
The proof of Proposition \ref{prop: continuity of w} is given in Section \ref{sec: proofs}. 
\begin{Pro}\label{prop: continuity of w}
$\Yc^\Psi$ verifies
\begin{enumerate}
    \item $\Yc^\Psi$ is 1-Lipschitz. For any $z,z'\in  \R$ \begin{equation}
        |\Yc^\Psi(z) - \Yc^\Psi(z')|\leq |z-z'|.
    \end{equation} 
\item $\Yc^\Psi$ is non-increasing. Thus if $\Yc^\Psi(z_0) = 0$ then $\Yc^\Psi(z) = 0$ for all $z\geq z_0$.
\end{enumerate}
\end{Pro}

\paragraph{}
Define the infimum of the zero level-set \begin{equation}\label{eq: definition Z K}
    \Zc^\Psi := 
\inf \{z\in\R\ |\ \Yc^\Psi(z) = 0\}.
\end{equation}

We prove a first result linking the auxiliary control problem with the original constrained problem. Solving this easier problem provides bounds on the value function, by making the constraint function vary.

\begin{Theorem}\label{th: main}
\begin{enumerate}
    \item If for some $z\in \R\  \exists\ \alpha\in\Ac\ \mathrm{s.t.}\ \widehat{g}(\P_{X_T^{\alpha}})\leq Z^{z,\alpha}_{T},\ \Psi(s,\P_{X_s^{\alpha}})\leq 0,\ \forall\ s\in[0,T] $  then $ \Yc^\Psi(z) = 0 $. 
    \item If $V^\Psi$ is finite then $\Yc^\Psi(V^\Psi) = 0 $. Thus $\Zc^\Psi \leq V^\Psi$.
    \item We have the upper bound 
    \begin{equation}
        V^\Psi \leq \inf_{ \varepsilon > 0} \Zc^{\Psi+\varepsilon}.
    \end{equation} 
\end{enumerate}
\end{Theorem}

To sum up, when $V^\Psi<+\infty$, Theorem \ref{th: main} provides the bounds

\begin{equation}\label{eq: bounds}
    \Zc^\Psi \leq V^\Psi \leq \inf_{ \varepsilon > 0} \Zc^{\Psi+\varepsilon } .
\end{equation}
The proof of Theorem \ref{th: main} is given in Section \ref{sec: proofs}.

\begin{rem}\label{rem: existence OC}
In the easier case where optimal controls exist for the auxiliary problem, as assumed in \cite{BPZ16}, similar arguments as in \cite{BPZ16} (and Section \ref{sec:noise})  directly prove that $\Zc^{\Psi} = V^{\Psi}$ and that an optimal control $\alpha^*$ associated to the auxiliary problem $\Yc^\Psi(V)$ is optimal for the original problem. 
However some difficulties arise when trying to remove this assumption. 
\end{rem}

\begin{rem}\label{rem: finite rhs}
If there exists $\varepsilon_0>0$ such that $V^{\Psi+\varepsilon_0} <+\infty$ then $\Zc^{\Psi+\varepsilon_0} \leq V^{\Psi+\varepsilon_0} < +\infty$ by Theorem \ref{th: main}. Thus the right-hand side of \eqref{eq: bounds} is finite. 

On the other hand, we need to be careful about the form of the constraint function if we want to use \eqref{eq: bounds}. There are cases in which one choice of $\Psi$ gives a infinite right-hand side in this equation but another representation of the constraint gives an finite right-hand side.
Let us  consider for instance a one-dimensional terminal constraint in law $\varphi(\P_{X_T^{\alpha}})\leq 0$, represented by 
\begin{equation*}\Psi:(t,\mu) \mapsto
	\varphi(\mu)\mathds{1}_{t=T}.
	\end{equation*}  
We see that the constraint $\Psi(t,\mu)+\varepsilon  \leq 0$ would never be verified for any $t<T$ and any $\varepsilon>0$,  hence $V^{\Psi+\varepsilon } =+\infty$ and $Z^{\Psi+\varepsilon } = +\infty$. 
In that case there would be a gap in \eqref{eq: bounds} and one wouldn't be able to conclude that $V^{\Psi}=Z^{\Psi}$.
\end{rem}

In view of the above example in Remark \ref{rem: finite rhs}, we introduce a modified constraint function in order to deal with discrete time constraints, and also with a.s. constraints. 
Given  a constraint function $\Psi(t,\mu)$, we define 
\begin{align}\label{eq: modified constraint}
    & \overline{\Psi}_\kappa(t,\mu):= \Psi(t,\mu) - \kappa   \mathds{1}_{\{\Psi(t,\mu) \leq 0\}}, 
\end{align}  
with $\kappa >0$. By observing that $\overline{\Psi}_\kappa(t,\mu)$ $\leq$ $0$ $\Leftrightarrow$ $\Psi(t,\mu)$ $\leq$ $0$,  it follows that 

\begin{equation}\label{psikappa}
	V^\Psi= V^{\overline{\Psi}_\kappa},\ \Yc^\Psi=\Yc^{\overline{\Psi}_\kappa},\ \Zc^\Psi=\Zc^{\overline{\Psi}_\kappa}.
\end{equation}

\begin{rem} \label{remVZ} 
Notice that by taking $\varepsilon_0 < \kappa$, and assuming that $V^\Psi<\infty$, 
we have $\Zc^{\overline{\Psi}_\kappa+\varepsilon_0 }$ $<$ $\infty$. Indeed, 
by applying Theorem \ref{th: main} to $\overline{\Psi}_\kappa$, we have  
$\Zc^{\overline{\Psi}_\kappa+\varepsilon_0 }$ $\leq$ 
$V^{\overline{\Psi}_\kappa+\varepsilon_0 }$.  Moreover, by observing that  
an admissible control for the original problem $V^\Psi$ is also admissible for the auxiliary problem with constraint function $\overline{\Psi}_\kappa+\varepsilon_0 $, by definition of $\overline{\Psi}_\kappa$, this implies that $V^{\overline{\Psi}_\kappa+\varepsilon_0 }$ $<$ $\infty$.
\end{rem}

\subsection{Representation of the value function}
Now we prove under some assumptions on the constraints 
the continuity property $\Zc^ {\overline{\Psi}_\kappa}$ $=$ $\inf_{ \varepsilon > 0} \Zc^{\overline{\Psi}_\kappa+\varepsilon}$ 
in order to obtain a characterization of the original value function $V^\Psi$. 
The result relies on convexity arguments, and holds true within the linear-convex model assumption: 

\vspace{3mm}

\begin{assumption}[Lin-Conv] \label{hyplinconv} 
The coefficients $b$ valued in $\R^d$, $\sigma$ $=$ $(\sigma_j)_{1\leq j\leq d}$ valued in $\R^{d\times d}$ of the controlled mean-field equation are in the linear form:
\begin{align*} 
b(t,x,a,\nu) &= \;  \beta(t) + B(t) x + C(t)  a +  \bar B(t)  \int x \nu(\di x,\di a) + \bar C(t) \int a \nu(\di x,\di a) \\
\sigma_j(t,x,a,\nu) &= \; \gamma_j(t)  + D_j(t) x + F_j(t) a + \bar D_j(t)  \int x \nu(\di x,\di a) + \bar F_j(t) \int a \nu(\di x,\di a), 
\end{align*}
for $(t,x,a,\nu)$ $\in$ $[0,T]\times\R^d\times A\times\Pc_2(\R^d\times A)$, with $A$ convex set of $\R^q$, and with  bounded measurable function  $\beta$, $\gamma_j$ $B$, $D_j$,  $C$, $F_j$, $\bar B$, $\bar D_j$, $\bar C$,  
and $\bar F_j$, $j$ $=$ $1,\ldots,d$, on $[0,T]$,  valued respectively on $\R^d$, 
$\R^d$, $\R^{d\times d}$, $\R^{d\times d}$, $\R^{d\times q}$, $\R^{d\times q}$, $\R^{d\times d}$, $\R^{d\times d}$, $\R^{d\times q}$, $\R^{d\times q}$. 
\vspace{1mm}

The cost functions $f$,  $g$ and the constraint functions  $\Psi_l$, $l$ $=$ $1,\ldots,k$, are in the form:
\begin{align*}
f(t,x,a,\nu) \; = \; \tilde f\big(t,x,a,\int x \nu(\di x,\di a),\int a \nu(\di x,\di a) \big) & \quad g(x,\mu) \; = \; \tilde g\big(x,\int x \mu(\di x) \big), \\
\Psi_l(t,\mu) \; = \; \tilde\Psi_l\big(t,\int x \mu(\di x) \big).  & 
\end{align*}
 for $(t,x,a,\mu,\nu)$ $\in$ $[0,T]\times\R^d\times A\times\Pc_2(\R^d)\times\Pc_2(\R^d\times A)$, where $\tilde f(t,.)$ is convex on $\R^d\times \R^q\times\R^d\times \R^q$, and $\tilde g$, $\tilde \Psi_l(t,.)$  are  convex on $\R^d\times\R^d$. 
 \end{assumption}
 
 \vspace{3mm}
 
\begin{lem}\label{lem: convexity of y}
Under  Assumption \ref{hyplinconv},  the function $(z,\varepsilon)\in\R\times\R\mapsto \Yc^{\Psi + \varepsilon }(z)$ is  convex.
\end{lem}
\noindent The proof of Lemma \ref{lem: convexity of y} is given in Section \ref{sec: proofs}. 

\begin{Pro}
Let Assumption \ref{hyplinconv} hold. Then, $\Yc^\Psi$ being convex, positive and non-increasing, if $\Zc^\Psi < \infty$ then $\Yc^\Psi$ is decreasing on $(-\infty,\Zc^\Psi]$ and $\Yc^\Psi(z) = 0$ on $[\Zc^\Psi,\infty).$
\end{Pro}
\begin{proof}
By contradiction, if $\Yc^\Psi(a) = \Yc^\Psi(b) > 0$ with $a<b$ 
then by monotonicity $\Yc^\Psi([a,b]) = \{\Yc^\Psi(a)\}$ and $0\in\partial \Yc^\Psi(a)$ thus $\Yc^\Psi(x)\geq \Yc^\Psi(a)>0,\ \forall x\in\R$  which is not the case because $\Zc^\Psi < \infty$. As a consequence, $\Yc^\Psi$ is decreasing. Then by continuity of $\Yc^\Psi$ and definition of $\Zc^\Psi$ we obtain $\Yc^\Psi(\Zc^\Psi)=0 $. 
\end{proof}

\begin{Theorem}\label{th: representation} 
Under Assumption \ref{hyplinconv},  assume  that $-\infty$ $<$ $V^\Psi$ $<$ $\infty$. Then we have the representation
\begin{equation*}
\Zc^{\Psi}  = V^{\Psi}.
\end{equation*}
 Moreover $\varepsilon$-optimal controls $\alpha^\varepsilon$ for the auxiliary problem $ \Yc^\Psi(V^\Psi)$ are $\varepsilon$-admissible $\varepsilon$-optimal controls for the original problem in the sense that 
\begin{align*}
    J(X_0,\alpha^\varepsilon) \leq  V^\Psi + \varepsilon,\ \sup_{0\leq s \leq T} \Psi(s,\P_{X_s^{\alpha^\varepsilon}}) \leq \varepsilon.
\end{align*} 
\end{Theorem}

\begin{proof}[Proof of Theorem \ref{th: representation}]
 We prove the continuity of $\Zc^{\overline{\Psi}_\kappa}$ along the curve $\Zc^{\overline{\Psi}_\kappa+\varepsilon}$ for $\varepsilon\in\R$ where $\overline{\Psi}_\kappa$ is defined in \eqref{eq: modified constraint}. 

Let $\kappa>0$ and $\varepsilon_0<\kappa$. By Remark \ref{remVZ}, we know that $\Zc^{\overline{\Psi}_\kappa+\varepsilon_0}<\infty$. We consider 
the function
\begin{align*} 
\Phi(\varepsilon) & := \; \Zc^{\overline{\Psi}_\kappa+\varepsilon } \; = \; 
\inf\{ z \in \R: \Yc^{\overline{\Psi}_\kappa+\varepsilon}(z)  \leq  0\} \; < \; \infty, \quad \varepsilon < \varepsilon_0, 
\end{align*}
and observe by direct verification that it is convex on $(-\infty,\varepsilon_0)$, using Lemma \ref{lem: convexity of y}.
Moreover, by \eqref{eq: bounds} applied to $\bar\Psi_\kappa$, and $\eqref{psikappa}$, we have 
$\Phi(\varepsilon)$ $=$ $\Zc^{\bar\Psi_\kappa+\varepsilon}$  $\geq$ $V^{\bar\Psi_\kappa}$ $=$ 
$V^{\Psi}$ $>$ $-\infty$.    
Therefore, $\Phi$ is a convex, and finite function on $(-\infty,\varepsilon_0)$, hence it is continuous (see e.g. 
Corollary 10.1.1 in  \cite{rockConvexAnalysis}), in particular at $\varepsilon$ $=$ $0$.    
As a consequence $\Zc^ {\overline{\Psi}_\kappa} = \inf_{ \varepsilon > 0} \Zc^{\overline{\Psi}_\kappa+\varepsilon }$, and by Theorem \ref{th: main} applied to $\overline{\Psi}_\kappa$, we obtain $\Zc^{\overline{\Psi}_\kappa} = V^{\overline{\Psi}_\kappa}$. Then recalling that $\Zc^\Psi = \Zc^{\overline{\Psi}_\kappa},\ V^\Psi = V^{\overline{\Psi}_\kappa}$, 
 the result follows.  
 \\
 Concerning the controls, take $\varepsilon$ $>$ $0$, and consider an $\varepsilon$-optimal control $\alpha^{\varepsilon}\in\Ac$ such that
\begin{equation*}
    \{\widehat{g}(\P_{T}^{\alpha^{\varepsilon}})- Z^{\Zc^{\Psi},\alpha^{\varepsilon}}_T \}_+ +  \sup_{s\in[0,T]} \{\Psi(s,\P_{X_s^{\alpha^{\varepsilon}}})\}_+ \leq \varepsilon.
\end{equation*}   The two terms on the l.h.s. being non-negative, they both are smaller than $\varepsilon$  and thus \begin{equation*}
    \widehat{g}(\P_{T}^{\alpha^{\varepsilon}})\leq 
    Z^{\Zc^{\Psi },\alpha^{\varepsilon}}_T +\varepsilon,\ \mathrm{and}\  \Psi(s,\P_{X_s^{\alpha^{\varepsilon}}})\leq  \varepsilon,\  \forall\ s\in[0,T].
\end{equation*}
Hence $   J(X_0,\alpha^{\varepsilon})\leq \Zc^{\Psi} + \varepsilon = V^{\Psi} + \varepsilon$ and 
$\Psi( s,\P_{X_s^{\alpha^{\varepsilon}}})\leq \varepsilon ,\  \forall\ s\in[0,T].$
\end{proof}

\subsection{Proofs}\label{sec: proofs}

\begin{proof}[Proof of Proposition \ref{prop: continuity of w}]
1) By the inequalities $|\inf_u A(u) - \inf_u B(u)|\leq \sup_u |A(u) -  B(u)|$, $|\sup_u A(u) - \sup_u B(u)|\leq \sup_u |A(u) -  B(u)|$ 
we obtain for any 
$z,z'\in\R$
\begin{align}
    & |\Yc^\Psi(z) - \Yc^\Psi(z')| \\ &= \Big|\inf_{\alpha\in\Ac} \Big[\{\widehat{g}(\P_{X_T^{\alpha}})- Z^{z,\alpha}_{T}\}_+ 
    + \sup_{s\in[t,T]} \{\Psi(s,\P_{X_s^{\alpha}})\}_+ \Big] \\ & - \inf_{\alpha\in\Ac} \Big[\{\widehat{g}(\P_{X_T^{\alpha}})- Z^{z',\alpha}_{T}\}_+ + \sup_{s\in[t,T]} \{\Psi(s,\P_{X_s^{\alpha}})\}_+ \Big]\Big| \\
    & \leq \sup_{\alpha\in\Ac} \Big| \{\widehat{g}(\P_{X_T^{\alpha}})- Z^{z,\alpha}_{T}\}_+ -  \{\widehat{g}(\P_{X_T^{\alpha}})- Z^{z',\alpha}_{T}\}_+  +  \sup_{s\in[t,T]} \{\Psi(s,\P_{X_s^{\alpha}})\}_+-  \sup_{s\in[t,T]} \{\Psi(s,\P_{X_s^{\alpha}})\}_+\Big|\\
    & \leq  \sup_{\alpha\in\Ac} \Big| Z^{z,\alpha}_{T} -  Z^{z',\alpha}_{T} \Big| = |z - z'|, 
\end{align} 
by 1-Lipschitz continuity of $x\mapsto \{x\}_+$. 

2) Denote by 
\begin{align} \label{defL} 
L^\Psi(z,\alpha) &= \;  
\{\widehat{g}(\P_{X_T^{\alpha}})- Z^{z,\alpha}_{T}\}_+ +  \sup_{s\in[0,T]} \{\Psi(s,\P_{X_s^{\alpha}})\}_+, 
\end{align}
so that $\Yc^\Psi(z)$ $=$ $\inf_{\alpha\in\Ac} L^\Psi(z,\alpha)$. Then, it is clear that 

\begin{equation}
    z \leq z' \implies L^\Psi(z',\alpha) \leq L^\Psi(z,\alpha)
\end{equation} 
hence by minimizing, the same monotonicity property holds also for the value function 
\begin{equation}
    z \leq z' \implies \Yc^\Psi(z') \leq \Yc^\Psi(z).
\end{equation}
\end{proof}

\begin{proof}[Proof of Theorem \ref{th: main}]

1) $\exists\ \alpha\in\Ac,\ \widehat{g}(\P_{X_T^{\alpha}})\leq Z^{z,\alpha}_{T}$ and 
$ \Psi(s,\P_{X_s^{\alpha}})\leq 0,\  \forall\ s\in[0,T]$. Therefore
\begin{equation}
    \{\widehat{g}(\P_{X_T^{\alpha}})- Z^{z,\alpha}_{T}\}_+ +  \sup_{s\in[0,T]} \{\Psi(s,\P_{X_s^{\alpha}})\}_+ = 0
\end{equation} and by non-negativity of $\Yc$ we obtain $\Yc^\Psi(z) = 0$  
\newline

2) By continuity of $\Yc$ (Proposition \ref{prop: continuity of w})  
and  1),
we obtain $\Yc^\Psi(V^\Psi) = 0 $ by taking admissible 
$\varepsilon$-optimal controls for the original problem and taking the limit $\varepsilon\rightarrow 0$. By definition of  $\Zc^\Psi$ the property is established.\newline

3) We assume that exists $\varepsilon_0>0$ such that $\Zc^{\Psi+\varepsilon_0 }<+\infty$. 
If it is not the case then $\inf_{\varepsilon>0} \Zc^{\Psi + \varepsilon } = + \infty$ and the inequality is verified.  
Let $0<\varepsilon<\varepsilon_0$ satisfying  $\Zc^{\Psi + \varepsilon } <\infty$. 
By continuity of $\Yc$ in the $z$ variable (Proposition \ref{prop: continuity of w}), $\Yc^{\Psi +\varepsilon }(\Zc^{\Psi+\varepsilon }) = 0 $. 
Then by definition of $\Yc^{\Psi +\varepsilon }$, for $0<\varepsilon'\leq\varepsilon, \exists\ \alpha^{\varepsilon'}\in\Ac$ such that
\begin{equation}
    \{\widehat{g}(\P_{T}^{\alpha^{\varepsilon'}})- 
    Z^{\Zc^{\Psi +\varepsilon },\alpha^{\varepsilon'}}_{T} \}_+ +  \sup_{s\in[0,T]} \{\Psi(s,\P_{X_s^{\alpha^{\varepsilon'}}})+ \varepsilon \}_+ \leq \varepsilon'.
\end{equation} The two terms on the l.h.s. being non-negative, they both are smaller than $\varepsilon'$ and thus \begin{equation}
    \widehat{g}(\P_{T}^{\alpha^{\varepsilon'}})\leq 
    Z^{\Zc^{\Psi +\varepsilon },\alpha^{\varepsilon'}}_{T} +\varepsilon',\ \mathrm{and}\  \Psi(s,\P_{X_s^{\alpha}})\leq \varepsilon'- \varepsilon \leq 0,\  \forall\ s\in[0,T].
\end{equation}
Hence \begin{equation}
   J(\alpha^{\varepsilon'})\leq \Zc^{\Psi + \varepsilon } + \varepsilon'
\end{equation} and 
\begin{equation}
\Psi( s,\P_{X_s^{\alpha}})\leq 0 ,\  \forall\ s\in[0,T].
\end{equation} 
Therefore by arbitrariness of $\varepsilon'$ verifying $0<\varepsilon'<\varepsilon$ we conclude that $V^\Psi \leq \Zc^{\Psi + \varepsilon } $. By arbitrariness of $\varepsilon$ verifying $0<\varepsilon<\varepsilon_0$ it follows 
\begin{equation}
    V^\Psi \leq \inf_{\varepsilon \in (0,\varepsilon_0)} \Zc^{\Psi + \varepsilon } = \inf_{\varepsilon > 0} \Zc^{\Psi + \varepsilon }, 
\end{equation} where the last equality comes from the non-increasing property of $\Zc^{\Psi + \varepsilon }$ w.r.t. $\varepsilon.$
\end{proof}

\begin{proof}[Proof of Lemma \ref{lem: convexity of y}]
Under Assumption \ref{hyplinconv} on the linear dynamics of the controlled state process, we have  for all $\alpha$ $\in$ $\Ac$, $z$ $\in$ $\R$, $\varepsilon$ $\in$ $\R_+$, 
\begin{align*}
L^{\Psi + \varepsilon}(z,\alpha) &= \;  \Big\{ \E \big[ \tilde g(X_T^\alpha,\E[X_T^\alpha]) + \int_0^T \tilde f(s,X_s^\alpha,\alpha_s,\E[X_s^\alpha],\E[\alpha_s])\ \di s \big]  - z \Big\}_+  \\
& \quad \quad \quad + \;  \sup_{s\in[0,T]} \big\{ \tilde\Psi(s,\E[X_s^\alpha]) + \varepsilon\big\}_+. 
\end{align*}
Let $\alpha^1$, $\alpha^2$ be two arbitrary controls in $\Ac$, $z^1$, $z^2$ $\in$ $\R$, $\varepsilon^1$, $\varepsilon^2$ $\in$ $\R_+$, and $\lambda$ $\in$ $(0,1)$.  Define $\alpha$ $=$ $\lambda\alpha^1 + (1-\lambda)\alpha^2$,  and notice by the linear mean-field dynamics in Assumption \ref{hyplinconv} that 
$X^\alpha$ $=$ $\lambda X^{\alpha^1} + (1-\lambda) X^{\alpha^2}$.  Then, by the convexity assumption on $\tilde f$, $\tilde g$, and $\tilde \Psi$ in Assumption \ref{hyplinconv},  and the convexity of $x\mapsto \{x\}_+$, we have 
\begin{align*}
L^{\Psi + \lambda \varepsilon^1  + (1-\lambda)\varepsilon^2 }(\lambda z^1 + (1-\lambda) z^2,\alpha) & \leq \;  \lambda  L^{\Psi+\varepsilon^1 }(z^1,\alpha^1) + (1-\lambda)  L^{\Psi +\varepsilon^2 }(z^2,\alpha^2).
\end{align*}
By taking the infimum over $\alpha^1$, $\alpha^2$ in the r.h.s. of the above inequality, we deduce the required  convexity result: 
\begin{align*}
\Yc^{\Psi + \lambda \varepsilon^1  + (1-\lambda) \varepsilon^2}( \lambda z^1 + (1-\lambda) z^2) & \leq \; L^{\Psi + \lambda \varepsilon^1  + (1-\lambda)\varepsilon^2 }(\lambda z^1 + (1-\lambda) z^2,\alpha)  \\
& \leq \;  \lambda \Yc^{\Psi + \varepsilon^1}(z^1)  + (1-\lambda) \Yc^{\Psi + \varepsilon^2}(z^2). 
\end{align*} 
\end{proof}
\subsection{Potential extension towards dynamic programming}

If one wants to use dynamic programming in order to solve the auxiliary control problem, it requires to write it down under a Markovian dynamic formulation.  Define
\begin{align}
    X_s^{t,\xi,\alpha} & = \xi + \int_t^s b\big(u,X_u^{t,\xi,\alpha},\alpha_u,\P_{(X_u^{t,\xi,\alpha},\alpha_u)})\ \di u+ \int_t^s \sigma(u,X_u^{t,\xi,\alpha},\alpha_u,\P_{(X_u^{t,\xi,\alpha},\alpha_u)}\big)\ \di W_u, \label{eq: MKV dynamic}
\end{align}
for $t$ $\in$ $[0,T]$, and $\xi\in L^2(\Fc_t,\R^d)$, and notice that we have the flow property  
\begin{equation}
X_r^{t,\xi,\alpha} \; = \;  X_r^{s,X_s^{t,\xi,\alpha},\alpha}, \quad 
\P_{X_r^{t,\xi,\alpha}} \; = \;  \P_{X_r^{s,X_s^{t,\xi,\alpha},\alpha}},\ \forall\ 0\leq s\leq r \leq T,
\end{equation} 
coming from existence and pathwise uniqueness in \eqref{eq: MKV SDE}. We thus consider 
the cost function
\begin{align}
J(t,\xi,\alpha)  :=\ & \E\Big[\int_t^T f\big(s,X_s^{t,\xi,\alpha},\alpha_s, \P_{(X_s^{t,\xi,\alpha},\alpha_s)}\big)\ \di s + g\big(X_T^{t,\xi,\alpha}, \P_{X_T^{t,\xi,\alpha}}\big) \Big], 
\end{align} 
and the value function
\begin{align}
V(t,\xi) := \inf_{\alpha\in\Ac} \{ J(t,\xi,\alpha)\ |\   \Psi(s,\P_{X_s^{t,\xi,\alpha}})\leq 0,\ \forall\ s\in[t,T]  \}. 
\end{align}  
Then we introduce the auxiliary state variable 
\begin{align}\label{eq: definition Z dynamic}
    Z^{t,\xi,z,\alpha}_r  := z - \E\Big[\int_t^r f\big(s,X_s^{t,\xi,\alpha},\alpha_s, \P_{(X_s^{t,\xi,\alpha},\alpha_s)}\big)\ \di s\Big]= z - \int_t^r \widehat{f}\big(s,\P_{(X_s^{t,\xi,\alpha},\alpha_s)}\big)\ \di s, \quad t \leq r \leq T, 
\end{align}
and the auxiliary value function  is given by
\begin{align}
\Yc^\Psi(t,\xi, z) 
& = \inf_{\alpha\in\Ac} \Big[\{\widehat{g}(\P_{X_T^{t,\xi,\alpha}})- Z^\alpha_{t,\xi,z}(T)\}_+ + \sup_{s\in[t,u]} \{\Psi(s,\P_{X_s^{t,\xi,\alpha}})\}_+\Big) \Big]\\ & = :   \inf_{\alpha\in\Ac} L^\Psi(t,\xi, z,\alpha).\label{eq: auxiliary problem dynamic}
\end{align}
We can treat the non-Markovian formulation of this problem by introducing as in \cite{BPZ15} an additional state variable $Y_u^{t,\xi,\alpha,m}=\Big( \sup_{s\in[t,u]} \{\Psi(s,\P_{X_s^{t,\xi,\alpha}})\}_+\Big)\vee m\geq0$ for $u\geq t$  with $m\in\R$ and the value function
\begin{align}
\tilde{\Yc}^\Psi(t,\xi, z,m)  
= &\inf_{\alpha\in\Ac} \Big[\{\widehat{g}(\P_{X_T^{t,\xi,\alpha}})- Z^\alpha_{t,\xi,z}(T)\}_+ + Y_T^{t,\xi,\alpha,m} \Big] =: \inf_{\alpha\in\Ac} \overline{L}^\Psi(t,\xi, z,m,\alpha).\label{eq: auxiliary problem bis dynamic}
\end{align} 
The two problems are related by
\begin{equation}
\Yc^\Psi(t,\xi, z)  = \tilde{\Yc}^\Psi(t,\xi, z, \{\Psi(t,\P_{X_t^{t,\xi,\alpha}})\}_+).
\end{equation}
With this formulation, the problem \eqref{eq: auxiliary problem dynamic} becomes a Mayer-type Markovian optimal control problem in the augmented state space $[0,T] \times L^2(\Fc_0,\R^d) \times \R \times \R$. As mentioned in \cite{BPZ15}, this procedure is used for instance for hedging lookback options in finance, see e.g. \cite{GHT14}.  Now the infimum of the zero level-set is given by \begin{equation}\label{eq: definition Z K bis}
    \Zc^\Psi(t,\xi) := \inf \{z\in\R\ |\ \tilde{\Yc}^\Psi(t,\xi,z,0) = 0\}.
\end{equation} Indeed note that $\tilde{\Yc}^\Psi(t,\xi, z,m) = 0 \iff m\leq0\ \mathrm{and}\  \tilde{\Yc}^\Psi(t,\xi, z,0)=0.$

\paragraph{}
The Lipschitz and convexity properties of the value function are proven exactly as in Proposition \ref{prop: continuity of w} but we detail here the continuity in space and in the running maximum variable $m$.

\begin{assumption}\label{assumption : lipschitz dynamic}
$\Psi,f,g,b,\sigma$ are Lipschitz continuous uniformly with respect to to other variables. Namely exists $[\Psi], [f], [g], [b], [\sigma], L >0$ and locally bounded functions $h,\ell,\mathfrak{L}:[0,+\infty) \mapsto [0,+\infty)$ such that for any $t\in[0,T],\ x,x'\in\R^d, \mu \in\Pc_2(\R^d), \nu,\nu'\in\Pc_2(\R^d\times A)$, $a\in A$
  \begin{align}
      |\Psi(t,\mu) - \Psi(t,\mu')| & \leq [\Psi] \Wc_2(\mu,\mu')\\
      |f(t,x,a,\nu) - f(t,x',a,\nu')| & \leq [f] (|x-x'|+\Wc_2(\nu,\nu'))\\
      |g(x,\mu) - g(x,\mu')| & \leq [g] (|x-x'|+\Wc_2(\mu,\mu'))\\
      |b(t,x,a,\nu) - b(t,x',a,\nu')| & \leq [b] (|x-x'| + \Wc_2(\nu,\nu'))\\
        |\sigma(t,x,a,\nu) - \sigma(t,x',a,\nu')| & \leq [\sigma] (|x-x'| + \Wc_2(\nu,\nu'))\\
       |b(t,0,a,\delta_{0}\otimes\mu)| + |\sigma(t,0,a,\delta_{0}\otimes\mu)| + |f(t,0,a,\delta_{0}\otimes\mu)| & \leq L\\
       |f(t,x,a,\nu)|  & \leq  h(\|\nu\|_2)(1+|x|^2)\\
        |g(x,\mu)| & \leq  \ell(\|\mu\|_2)(1+|x|^2)\\
       |\Psi(t,\mu)| & \leq \mathfrak{L}(\|\mu\|_2).
  \end{align}
\end{assumption}
\begin{Pro}\label{pro: continuity dynamic}
  Under Assumption \ref{assumption : lipschitz dynamic} $\tilde{\Yc}^\Psi$ is Lipschitz continuous: there exists $C>0$ such that for any $t\in[0,T],\ \xi,\xi'\in L^2(\Fc_t,\R^d), m,m'\in\R$
  \begin{equation}
      |\tilde{\Yc}^\Psi(t,\xi,z,m) - \tilde{\Yc}^\Psi(t,\xi',z',m')| \leq | z  -  z' |+ |m-m'| + C \sqrt{\E|\xi-\xi'|^2}.
  \end{equation}
\end{Pro}

\begin{proof}[Proof of Proposition \ref{pro: continuity dynamic}]
	By the inequalities $|\inf_u A(u) - \inf_u B(u)|\leq \sup_u |A(u) -  B(u)|$, $|\sup_u A(u) - \sup_u B(u)|\leq \sup_u |A(u) -  B(u)|$, and $|a\vee b - c\vee d|\leq |a-c|\vee |b-d| \leq |a-c| + |b-d|$  we obtain for any $\xi,\xi'\in L^2(\Fc_t,\R^d)$  (if $\Psi$ is not continuous consider $\xi=\xi'$) 
	\begin{align}
	& |\tilde{\Yc}^\Psi(t,\xi,z,m) - \tilde{\Yc}^\Psi(t,\xi',z',m')| \\ 
	& \leq \sup_{\alpha\in\Ac} | \{\widehat{g}(\P_{X_T^{t,\xi,\alpha}})- Z^{t,\xi,z,\alpha}_{T} \}_+ -  \{\widehat{g}(\P_{X_T^{t,\xi',\alpha}})- Z^{t,\xi',z',\alpha}_{T} \}_+  + 
	\\ & 
	\sup_{s\in[t,T]} \{\Psi(s,\P_{X_s^{t,\xi,\alpha}})\}_+\vee m - \sup_{s\in[t,T]} \{\Psi(s,\P_{X_s^{t,\xi',\alpha}})\}_+\vee m' |\\
	& \leq \sup_{\alpha\in\Ac} (| \widehat{g}(\P_{X_T^{t,\xi,\alpha}})- \widehat{g}(\P_{X_T^{t,\xi',\alpha}})| + | Z^{t,\xi,z,\alpha}_{T} 
	-  Z^{t,\xi',z',\alpha}_{T} | + | \sup_{s\in[t,T]} \{\Psi(s,\P_{X_s^{t,\xi,\alpha}})\}_+ - \sup_{s\in[t,T]}\{ \Psi(s,\P_{X_s^{t,\xi',\alpha}})\}_+ |\\ & +|m-m'|)
	\\
	& \leq \sup_{\alpha\in\Ac} \Big(|\E[g(X_T^{t,\xi,\alpha},\P_{X_T^{t,\xi,\alpha}}) - g(X_T^{t,\xi',\alpha},\P_{X_T^{t,\xi',\alpha}})] | 
	+ | z  -  z' | \\ &+ \Big| \E\Big[\int_t^T f\big(s,X_s^{t,\xi,\alpha},\alpha_s,\P_{(X_s^{t,\xi,\alpha},\alpha_s)}\big)\ \di s -  \int_t^T f\big(s,X_s^{t,\xi',\alpha},\alpha_s,\P_{(X_s^{t,\xi',\alpha},\alpha_s)}\big)\ \di s \Big]\Big|\Big) \\ & + \sup_{\alpha\in\Ac} \sup_{s\in[t,T]} |  \{\Psi(s,\P_{X_s^{t,\xi,\alpha}})\}_+ -  \{\Psi(s,\P_{X_s^{t,\xi',\alpha}})\}_+ |+|m-m'|\\
	&  \leq [\widehat{g}] \sup_{\alpha\in\Ac} ( \E|X_T^{t,\xi,\alpha}-X_T^{t,\xi',\alpha}| +  \Wc_2(\P_{X_T^{t,\xi,\alpha}},\P_{X_T^{t,\xi',\alpha}})) + | z  -  z' |+ |m-m'| + [\Psi] \sup_{\alpha\in\Ac} \sup_{s\in[t,T]} \Wc_2(\P_{X_s^{t,\xi,\alpha}},\P_{X_s^{t,\xi',\alpha}})   \\ &+ T [f] \sup_{\alpha\in\Ac} \{   \E[\sup_{s\in[t,T]} |X_s^{t,\xi,\alpha}-X_s^{t,\xi',\alpha}|] + \sup_{s\in[t,T]}\Wc_2(\P_{X_s^{t,\xi,\alpha}},\P_{X_s^{t,\xi',\alpha}}) \} ,
	\end{align} 
	by Lipschitz continuity of $\Psi, x\mapsto \{x\}_+$. We recall the estimates \begin{align}
	    & \sup_{s\in[t,T]} \Wc_2(\P_{X_s^{t,\xi,\alpha}},\P_{X_s^{t,\xi',\alpha}})  = \sqrt{\sup_{s\in[t,T]} \Wc_2(\P_{X_s^{t,\xi,\alpha}},\P_{X_s^{t,\xi',\alpha}})^2} \leq C\sqrt{\E|\xi - \xi' |^2}\\
	    &  \E[\sup_{s\in[t,T]}|X_s^{t,\xi,\alpha}-X_s^{t,\xi',\alpha}| ] \leq C\E|\xi - \xi'|\leq C \sqrt{\E|\xi-\xi'|^2},
 	\end{align} obtained by standard arguments (see e.g. the proof of Proposition 3.3 in \cite{CP19}). Then the result follows.
	
\end{proof}

\begin{Pro}[Law invariance properties]\label{prop: law invariance} 
	Under Assumption \ref{assumption : lipschitz dynamic}, we have law invariance of $\tilde{\Yc}^\Psi $ and $\Zc^\Psi$, namely if $\xi, \eta$ are $\Fc_t$-adapted square integrable with the same law, for any $(t,z,m)\in[0,T]\times\R\times\R$
	\begin{align}
	& \tilde{\Yc}^\Psi(t,\xi,z,m) = \tilde{\Yc}^\Psi(t,\eta,z,m)\\
	& \Zc^\Psi(t,\xi) = \Zc^\Psi(t,\eta).
	\end{align}
	Therefore we can define the lifted functions $y^\Psi, z^\Psi$ on $[0,T]\times\Pc_2(\R^d)\times \R$ (respectively $[0,T]\times\Pc_2(\R^d)$ by  $y^\Psi(t,\P_\xi,z,m) := \tilde{\Yc}^\Psi(t,\xi,z,m)$ and $z^\Psi(t,\P_\xi,z,m) := \Zc^\Psi(t,\xi,z,m)$.
\end{Pro}

\begin{proof}
	Apply the same arguments as in Theorem 3.5. from \cite{CGKPR20} to the unconstrained Markovian value function $\Yc^\Psi$ on the extended state space. 
	In particular use the continuity of $\Yc^\Psi$ from Proposition \ref{prop: continuity of w}  
	and notice for a given control $\alpha$ that in Step 1 of Theorem 3.5.  from \cite{CGKPR20} the equality in law
	\begin{align}
	& ((X_s^{t,\xi,\alpha})_{s\in[t,T]}, (Z^{t,\xi,z,\alpha}_{s})_{s\in[t,T]},(Y_u^{t,\xi,\alpha,m})_{s\in[t,T]}, (\alpha_s)_{s\in[t,T]}) \\  \overset{\mathcal{L}}{=}\ & ((X_s^{t,\eta,\beta})_{s\in[t,T]}, (Z^{\beta,t,\eta,z})_{s\in[t,T]},(Y_u^{t,\eta,\beta,m})_{s\in[t,T]}, (\beta_s)_{s\in[t,T]}),
	\end{align} 
	holds true with $a_s $ defined in Lemma B.2. from \cite{CGKPR20} (verifying in particular the equality in law $\alpha_s = a_s(\xi,U_\xi)$) and $\beta_s = a_s(\eta,U_\eta) $  where $U_\eta$ (respectively $U_\xi$) is a $\Fc_t$-adapted uniform random variable on $[0,1]$ independent of $\eta$ (respectively $\xi$).
	Then use the definition \eqref{eq: definition Z K} to obtain the same law invariance property for $\Zc^\Psi$ too.
\end{proof}

Theorem \ref{th: main} and Theorem \ref{th: representation} are still valid in the the dynamic case, by applying the exact same arguments. More precisely for any $(t,\xi)\in [0,T]\times L^2(\Fc_t,\R^d)$, if $V^\Psi(t,\xi)<\infty$ then
\begin{equation}
    \Zc^\Psi(t,\xi) \leq V^\Psi(t,\xi) \leq \inf_{ \varepsilon > 0} \Zc^{\Psi+\varepsilon }(t,\xi) .
\end{equation} Similarly, arguments like in Theorem \ref{th: representation} prove that
\begin{equation}
    \Zc^\Psi(t,\xi) = V^\Psi(t,\xi),
\end{equation} if $V^\Psi(t,\xi)<\infty$.

If the value function is law invariant (see Proposition \ref{prop: law invariance}) and Theorem \ref{th: representation} holds true, we expect $y$ to be formally (by combining arguments from \cite{BPZ15,CP19}) characterized by a Master Bellman equation in Wassertein space with oblique derivative boundary conditions.

\section{An alternative auxiliary problem}\label{sec: integral}

We study the constrained McKean-Vlasov control problem
\begin{align}
    V & : = \;  \inf_{\alpha\in\Ac} 
    \big\{J(X_0,\alpha)  :  \Psi(t,\P_{X_t^{\alpha}})\leq 0,\ \forall\ t\in[0,T],\ 
    \varphi(\P_{X_T^\alpha}) \leq 0 \big\}, 
\end{align}
where we now assume that the running constraint $\Psi$ is continuous 
(hence, no discrete time constraints, see Remark \ref{rem: discrete}), and with a terminal 
constraint function $\varphi$. 
We now consider an alternative auxiliary control problem as in \cite{BPZ16}: 
\begin{align}\label{eq: auxiliary problem int}
  w(z) := \inf_{\alpha\in\Ac} \Big[\{\widehat{g}(\P_{X_T^{\alpha}})- Z^{z,\alpha}_{T} \}_+ +  \int_0^T \{\Psi(s,\P_{X_s^{\alpha}})\}_+\ \di s +\{\varphi(\P_{X_T^{\alpha}})\}_+ \Big]. 
\end{align}
Compared to the control problem \eqref{eq: auxiliary problem} of the previous section, the    penalization term of the constrained function $\Psi$ is in integral form instead of a supremum form.  It follows that this problem is is Markovian with respect to the variables $X_t$, $\P_{X_t}$ and $Z_t$, and we shall show that it also provides a similar representation of the value function by its zero level set: 
\begin{align}
V &= \; \inf\{ z \in \R: w(z) = 0 \},    
\end{align}
but under the additional assumption that optimal controls do exist. Actually, we prove this result in the more general case with common noise in the next section. 

The mean-field  control problem \eqref{eq: auxiliary problem int} is Markovian with respect to 
the state variables $(X_t^\alpha$, $\P_{X_t^\alpha}$, $Z^{z,\alpha}_t)$, and it is known from  \cite{CGKPR20} that the infimum  
over open-loop controls $\alpha$ in $\Ac$ can be taken equivalently over randomized feedback policies, i.e. controls $\alpha$ in the form: $\alpha_t$ $=$ $\mfa(t,X_t^\alpha,\P_{X_t^\alpha},Z_t^{z,\alpha},U)$, for some deterministic function $\mfa$ from 
$[0,T]\times\R^d\times\Pc(\R^d)\times\R\times[0,1]$ into $A$, where $U$ is an $\Fc_0$-measurable uniform random variable on $[0,1]$. 

Let us now discuss conditions under which the infimum in \eqref{eq: auxiliary problem int} can be taken equivalently over 
(deterministic) feedback policies, i.e. for controls $\alpha$ in the form: $\alpha_t$ $=$ $\mfa(t,X_t^\alpha,\P_{X_t^\alpha},Z^{z,\alpha}_t)$, for some deterministic function $\mfa$ from 
$[0,T]\times\R^d\times\Pc(\R^d)\times\R$ into $A$. This will be helpful for numerical purpose in Section \ref{sec:numerical}. 
We assume on top of Assumption \ref{assumption : lipschitz dynamic} that the running cost $f$, the drift $b$ and the volatility coefficient $\sigma$ do not depend on the law of the control process. We also assume that the running cost $f$ $=$ $f(t,x,\mu)$ does not depend on the control argument. The terminal constraint function $\varphi$ should also verify the same assumptions as the terminal cost function $g$, namely Lipschitz continuity and local boundedness (see Assumption \ref{assumption : lipschitz dynamic}).

In this case, the corresponding dynamic auxiliary problem of \eqref{eq: auxiliary problem int} is written as 
\begin{align}\label{eq: aux closed loop}
\mrw(t,\mu,z) &= \inf_{\alpha\in\Ac}\  \Big[\{\widehat{g}(\P_{X_T^{t,\xi,\alpha}})- Z^{t,\xi,z,\alpha}_{T}\}_+ +   \int_0^T \{\Psi(s,\P_{X_s^{t,\xi,\alpha}})\}_+\ \di s + \{\varphi(\P_{X_T^{t,\xi,\alpha}} )\}_+ \Big]\\
X_r^{t,\xi,\alpha}  &= \xi + \int_t^r b\big(s,X_s^{t,\xi,\alpha},\alpha_s,\P_{X_s^{t,\xi,\alpha}})\ \di s + \int_t^r \sigma\big(s,X_s^{t,\xi,\alpha},\alpha_t,\P_{X_s^{t,\xi,\alpha}}\big)\ \di W_s, \quad \xi \sim \mu,  \\
Z^{t,\xi,z,\alpha}_{r} &= z - \int_t^r \bar{f}\big(s,\P_{X_s^{t,\xi,\alpha}}\big)\ \di s, \quad r\geq t, 
\end{align}
where  $\bar f$ is the function defined on  $[0,T]\times\Pc_2(\R^d)$ by  $\bar f(t,\mu)$ $=$  
$\int_{\R^d} f\big(t,x, \mu\big)\ \mu(\di x)$. Note that we have applied Theorem 3.5 from \cite{CGKPR20} to obtain the law invariance of the auxiliary value function which can be written as a function of the measure $\mu$.
From Theorem 3.5, Proposition 5.6. 2), and equation (5.17) in \cite{CGKPR20} (see also Remark 5.2. from \cite{CP19} and Section 6 in \cite{PW18}) we see that the Bellman equation for problem \eqref{eq: aux closed loop} is:
\begin{equation}\label{eq: Bellman}
    \begin{cases}
        & \partial_t \mrw(t,\mu,z)  + \E[\inf_{a\in A} \{ b\big(t,\xi,a,\mu)\partial_\mu \mrw(t,\mu,z)(\xi) - \bar f(t,\mu)\partial_z \mrw(t,\mu,z) \\ & \ \ +\frac{1}{2}  \mathrm{Tr} (\sigma\sigma^\top\big(t,\xi,a,\mu)\partial_x\partial_\mu \mrw(t,\mu,z)(\xi) ) \}] 
        +  \{\Psi(t,\mu)\}_+ = 0 \ \mathrm{for}\ (t,\mu,z)\in[0,T]\times\Pc_2(\R^d)\times \R\\
        & \mrw(T,\mu,z) = \{\widehat{g}(\mu)- z\}_+ + \{\varphi(\mu )\}_+ \ \mathrm{for}\ (\mu,z)\in\Pc_2(\R^d)\times \R.
    \end{cases}
\end{equation} 

By assuming that $\mrw$ is a smooth solution to this Bellman equation, and when   the infimum in 
\begin{equation}
  \inf_{a\in A}  \{ b\big(t,x,a,\mu)\partial_\mu \mrw(t,\mu,z)(x) - \bar f(t,\mu)\partial_z \mrw(t,\mu,z)  +\frac{1}{2}  \mathrm{Tr} (\sigma\sigma^\top\big(t,x,a,\mu)\partial_x\partial_\mu \mrw(t,\mu,z)(x) ) \}  
\end{equation} 
is attained for some measurable function $\hat\mfa(t,x,\mu,z)$ on $[0,T]\times\R^d\times\Pc(\R^d)\times\R$, we get an optimal control for \eqref{eq: auxiliary problem int} given in feedback form by $\alpha_t^*$ $=$ $\hat\mfa(t,X_t^{\alpha^*},\P_{X_t^{\alpha^*}},Z_t^{z,\alpha^*})$, $0\leq t\leq T$,  which shows that one can restrict in \eqref{eq: auxiliary problem int} to deterministic feedback policies.

\section{Extension to the common noise setting} \label{sec:noise}

We briefly discuss how the state constraints can be extended to mean-field control problems with common noise. In this case, in contrast with the previous section, we need to assume the existence of optimal control for the auxiliary unconstrained problem. It is similar to the assumption made by \cite{BPZ16}. Let $W^0$ be a $p$-dimensional Brownian motion independent 
of $W$, and denote by $\F^0 = (\Fc_t^0)_t$ the filtration generated by $W^0$. 
We consider the following cost and dynamics:
\begin{align}\label{eq: original problem common}
 J(\alpha) & = \E\Big[\int_0^T f\big(t,X_t^{\alpha},\alpha_t, \P_{(X_t^{\alpha},\alpha_t)}^{W^0}\big)\ \di t + g\big(X_T^{\alpha}, \P_{X_T}^{W^0}\big) \Big]\\
    \di X_t^{\alpha} & = b\big(t,X_t^{\alpha},\alpha_t,\P_{(X_t^{\alpha},\alpha_t)}^{W^0}\big)\ \di t + \sigma\big(t,X_t^{\alpha},\alpha_t,\P_{(X_t^{\alpha},\alpha_t)}^{W^0}\big)\ \di W_t + \sigma^0\big(t,X_t^{\alpha},\alpha_t,\P_{(X_t^{\alpha},\alpha_t)}^{W^0}\big)\ \di W_t^0,
\end{align}
where $\P_{(X_t^{\alpha},\alpha_t)}^{W^0}$ is the joint conditional law of $(X_t^{\alpha},\alpha_t)$ given $W^0$. The control process $\alpha$ belongs to a set $\Ac$ of $\F$-progressively measurable processes with values in a set $A\subset\R^q$. 

The controlled McKean-Vlasov process $X$ is constrained to verify $\Psi(t,\P_{X_t^{\alpha}}^{W_0})\leq 0$ and $ \varphi(\P_{X_T^{\alpha}}^{W_0}) \leq 0$. 
The proofs still follow the arguments from \cite{BPZ16} but are slightly more involved than in Section \ref{sec: no common noise} due to the additional noise appearing in the conditional law with respect to the common noise. We refer to \cite{PW17, DPT19} for the dynamic programming approach to these problems. The problem of interest is
\begin{align}
    V^0 = \inf_{\alpha\in\Ac} \{J(\alpha)\ |\   \Psi(t,\P_{X_t^{\alpha}}^{W_0})\leq 0 ,\ \forall\ t\in[0,T],\ \varphi(\P_{X_T^{\alpha}}^{W_0}) \leq 0 \  \}.
\end{align}

\subsection{Representation by a stochastic target problem and an associated control problem}

Given $z$ $\in$ $\R$, $\alpha$ $\in$ $\Ac$, and $\beta$ $\in$ $L^2(\F^0,\R^p)$, the set of $\R^p$-valued $\F^0$-adapted processes $\beta$ s.t. $\E[\int_0^T |\beta_t|^2 \di t]$ $<$ $\infty$, define 
\begin{align}\label{eq: definition Z common noise}
    Z^{z,\alpha,\beta}_{t} :=  z - 
    \int_0^t \widehat f(s,\P_{(X_s^{\alpha},\alpha_s)}^{W^0})\ \di s 
    + \int_0^t \beta_s\ \di W_s^0, \quad 0 \leq t \leq T. 
\end{align}   
\begin{lem} \label{lem: target common}
The value function admits the \textbf{stochastic target problem} representation
\begin{align}
     V^0 = \inf \{z\in\R\ |\ & \exists\ (\alpha,\beta)\in\Ac\times L^2(\F^0,\R^p)\ \mathrm{s.t.}\ \widehat{g}(\P_{X_T^{\alpha}}^{W^0})\leq Z_T^{z,\alpha,\beta},\\ & \Psi(t,\P_{X_t^{\alpha}}^{W^0} )\leq 0,\ \forall\ t\in[0,T],\ \varphi(\P_{X_T^{\alpha}}^{W^0} )\leq 0,\ \P\ \mathrm{a.s.}   \}.
\end{align}
\end{lem}
Lemma \ref{lem: target common} is proven in Section \ref{sec: proofs common}. 

Define the \textbf{auxiliary unconstrained} control problem
\begin{align}\label{eq: auxiliary problem common}
  \Uc(z) := \inf_{(\alpha,\beta)\in\Ac\times L^2(\F^0,\R^p)} \E \Big[\{\widehat{g}(\P_{X_T^{\alpha}}^{W^0})- Z_T^{z,\alpha,\beta} \}_+ + \int_0^T\ \{\Psi(s,\P_{X_s^{\alpha}}^{W^0})\}_+\ \di s  + \{\varphi(\P_{X_T^{\alpha}}^{W^0} )\}_+  \Big]
\end{align} for $z\in \R$. We notice that $\Uc( z) \geq 0$. 

\begin{Pro}\label{pro: cont y common noise}
$\Uc$ is 1-Lipschitz. For any $z,z'\in  \R$ 
    \begin{equation}
        |\Uc(z) - \Uc(z')|\leq |z-z'|.
\end{equation}
\end{Pro}
Proposition \eqref{pro: cont y common noise} is proven exactly as \eqref{prop: continuity of w}.

\begin{assumption}\label{assumption: optimal controls common}
Problem \eqref{eq: auxiliary problem common} admits an optimal control for any $z\in  \R$ and the constraint function $(t,\mu)\in[0,T]\times\Pc_2(\R^d) \mapsto \Psi(t,\mu)$ is continuous.

\end{assumption}

\begin{rem}\label{rem: discrete}
Notice that the integral penalization in \eqref{eq: auxiliary problem common} does not allow to consider discrete times constraints (except at terminal time) because the contribution to the integral would be null and the constraint function $\Psi$ would be discontinuous in time. We could consider discrete time constraints in the objective of the auxiliary problem by adding a sum of functions of $\P_{X_{t_i}^{\alpha}}^{W^0}$ for some $(t_i)_{i}\in[0,T]$ but it would lose its standard Bolza form.
\end{rem}

\vspace{1mm}

Define $\Zc = \inf \{z\in\R\ |\ \Uc(z) = 0 \}$. 

\begin{Theorem}\label{th: main common noise}
\begin{enumerate}
    \item 
   If  $\exists\ (\alpha,\beta)\in\Ac\times L^2(\F^0,\R^p),\ \widehat{g}(\P_{X_T^{\alpha}}^{W^0})\leq Z^{z,\alpha,\beta}_{T}$,  $ \Psi(s,\P_{X_s^{\alpha}}^{W^0})\leq 0,\  \forall\ s\in[0,T],$ and $\varphi(\P_{X_T^{\alpha}}^{W^0} )\leq 0,\ \P$ a.s. then $\Uc(z)=0$. Hence $\Zc \leq V^0$.
    \item The value function verifies $ V^0\leq \Zc$ thus $ V^0 = \Zc$. Moreover optimal controls for the problem $\Uc(\Zc) = 0$ are optimal for the original problem.
\end{enumerate}
\end{Theorem}
Theorem \ref{th: main common noise} is proven in Section \ref{sec: proofs common}.

\subsection{Proofs in the common noise framework}\label{sec: proofs common}

\begin{proof}[Proof of Lemma \ref{lem: target common}]
We first observe that 
\begin{align}
    V^0 = \inf \{ z\in\R\ |\ & \exists\ \alpha\in\Ac\ \mathrm{s.t.}\ 
   \E \Big[  \int_0^T  \widehat f(s,\P_{(X_s^{\alpha},\alpha_s)}^{W^0}) \ \di s + \widehat{g}(\P_{X_T^{\alpha}}^{W^0})\Big]\leq z,\\ & \Psi(s,\P_{X_s^{\alpha}}^{W^0})\leq 0,\ \forall\ s\in[0,T],\ \P^0\ \mathrm{a.s.}  \}.
\end{align} 
We need to prove that for $z\in\R$ \begin{equation}\label{eq: target problem common noise}
    \exists\ (\alpha,\beta)\in\Ac\times L^2(\F^0,\R^p)\ \mathrm{s.t.}\ \widehat{g}(\P_{X_T^{\alpha}}^{W^0})\leq Z^{z,\alpha,\beta}_{T},\ \Psi(s,\P_{X_s^{\alpha}}^{W^0})\leq 0,\ \forall\ s\in[0,T],\ \varphi(\P_{X_T^{\alpha}}^{W^0} )\leq 0,\ \P^0\ \mathrm{a.s.},
\end{equation}  and 
\begin{equation}\label{eq: target common noise}
    \exists\ \alpha\in\Ac\ \mathrm{s.t.}\ \E \Big[\int_0^T 
    \widehat{f}(s,\P_{(X_s^{\alpha},\alpha_s)}^{W^0}) \ \di s + \widehat{g}(\P_{X_T^{\alpha}}^{W^0})\Big]\leq z,\ \Psi(s,\P_{X_s^{\alpha}}^{W^0})\leq 0,\ \forall\ s\in[0,T],\ \varphi(\P_{X_T^{\alpha}}^{W^0} )\leq 0,\ \P^0\ \mathrm{a.s.},
\end{equation}  are equivalent. It is immediate to see that
\eqref{eq: target problem common noise} $\implies$ \eqref{eq: target common noise} by taking the expectation and noticing that the Itô integral is a true martingale. Conversely, assuming \eqref{eq: target common noise}, the martingale representation theorem provides a process $\widehat{\beta}$ such that 
\begin{align}
 z & \geq  \E\Big[\int_0^T \widehat{f}(s,\P_{(X_s^{\alpha},\alpha_s)}^{W^0})\ \di s + \widehat{g}(\P_{X_T^{\alpha}}^{W^0})\Big]\\ & =   
 \int_0^T \widehat{f}(s,\P_{(X_s^{\alpha},\alpha_s)}^{W^0}) \ \di s + \widehat{g}(\P_{X_T^{\alpha}}^{W^0}) 
 - \int_0^T \widehat{\beta_s}\ \di W_s^0.
\end{align} Thus by \eqref{eq: definition Z common noise}
\begin{equation}
    Z^{z,\alpha,\widehat{\beta}}_T \geq \widehat{g}(\P_{X_T^{\alpha}}^{W^0}),\ \P^0\ \mathrm{a.s.},
\end{equation} and we see that \eqref{eq: target common noise}
 $\implies$ \eqref{eq: target problem common noise}. Then the result follows.
\end{proof}

\begin{proof}[Proof of Theorem \ref{th: main common noise}]

1) $\exists\ (\alpha,\beta)\in\Ac\times L^2(\F^0,\R^p),\ \widehat{g}(\P_{X_T^{\alpha}}^{W^0})\leq Z^{z,\alpha,\beta}_{T}$, 
$ \Psi(s,\P_{X_s^{\alpha}}^{W^0})\leq 0,\  \forall\ s\in[0,T]$ and $\varphi(\P_{X_T^{\alpha}}^{W^0} )\leq 0,\ \P^0$ a.s.. Therefore
\begin{equation}
    \{\widehat{g}(\P_{X_T^{\alpha}}^{W^0})- Z^{z,\alpha,\beta}_{T}\}_+ + \int_0^T\ \{\Psi(s,\P_{X_s^{\alpha}}^{W^0})\}_+\ \di s + \{\varphi(\P_{X_T^{\alpha}}^{W^0} )\}_+= 0,\ \P^0\ \mathrm{a.s.} 
\end{equation} and by non-negativity of $\Uc$  we obtain $\Uc(z) = 0$. 
Then with optimal controls $\alpha^*, \beta^*$ we obtain $\Uc(V^0) = 0 $. By definition of  $\Zc$ the property is established.\newline

2) By 1) and the continuity given by Proposition \ref{pro: cont y common noise}, we obtain $\Uc(\Zc) = 0 $.  Then by Assumption \ref{assumption: optimal controls common} $\exists\ (\alpha,\beta)\in\Ac\times L^2(\F^0,\R^p)$ such that
\begin{equation}
   \E^0\Big[ \{\widehat{g}(\P_{X_T^{\alpha}}^{W^0})- Z^{\Zc,\alpha,\beta}_{T}\}_+ + \int_0^T\ \{\Psi(s,\P_{X_s^{\alpha}}^{W^0})\}_+\ \di s + \{\varphi(\P_{X_T^{\alpha}}^{W^0} )\}_+ \Big] = 0.
\end{equation} The three terms on the l.h.s. being non-negative, they are in fact null $\P$ a.s. Thus \begin{equation}
    (\P_{X_T^{\alpha}}^{W^0}, Z^{\Zc,\alpha,\beta}_{T})\in\mathrm{Epi}(\widehat{g}),\   \Psi(s,\P_{X_s^{\alpha}}^{W^0})\leq 0\ \forall\ s\in[0,T],\ \mathrm{and}\, \varphi(\P_{X_T^{\alpha}}^{W^0} )\leq 0\ \  \,\ \P\ \mathrm{a.s.} 
\end{equation} 
by continuity of $\Psi$ and of $s\mapsto \P_{X_s^{\alpha}}^{W^0}$, which means $V^0 \leq \Zc$. By 1) it yields $V^0 = \Zc$. As a consequence the previous proof provides an optimal control $\alpha$ for the original problem. 
\end{proof}

\section{Applications and numerical tests} \label{sec:numerical} 

We design several machine learning methods to solve this problem. We discretize the problem in time, parametrize the control by a neural network and directly minimize the cost. When the constraints are almost sure, we can sometimes enforce them by choosing an appropriate neural network architecture, for instance in storage problems. A more adaptive alternative is to solve the unconstrained auxiliary problem. We propose an extension of the first algorithm from \cite{CL19} to achieve this task. Thus we obtain a machine learning method able to solve state constrained mean field control problems. 

\subsection{Algorithms}

We solve the auxiliary problem in the simpler case without common noise with a first algorithm. We fix a relevant line segment $K$ of $\R$   on which we are going to explore the potential values of the problem. For instance we know that the value is greater than the value of the unconstrained problem $V$ therefore it is useless to compute the auxiliary value function for $z\leq V$. We discretize the problem in time on the grid $t_k := \frac{kT}{N}$. We call $\Delta t := \frac{kT}{N}$ and the Brownian increment $ \Delta W_i := W_{t_{i+1}} - W_{t_{i}}$. For $j=1,\cdots,N$,  $ \Delta W_i^j$ (respectively $X_0^j$) correspond to samples from $N$ independent Brownian motions $W^j$ (respectively from $N$ independent random variables with law $\mu_0$).
For training we discretize $K$ by using $N$ points.
We choose $\varepsilon$ as a small parameter, typically smaller than $10^{-8}$. In our tests we took $\varepsilon = 10^{-8}$ but we notice that with the level of discretization we chose for $K$, with 25 points, the obtained values for $w$ decrease to around $10^{-5}$ and $10^{-6}$ before reaching exactly zero so any value of $\varepsilon < 10^{-8}$ yields the same result on our examples.  We refer to \cite{boka} for results on the numerical approximation of level sets with a given threshold in the context of constrained deterministic optimal control. We propose the following extension of the Method 1 from \cite{CL19}. It is tested in Subsection \ref{sec: mean variance}. It can indeed also be used to solve unconstrained problem.

\begin{rem}\label{rem: lambda}

We point out that adding an additional parameter $\Lambda>0$ in front of the constraint function does not modify the representation results. In that case we solve the following auxiliary problem  
\begin{align} \label{eq:algo1} 
& \;  \Yc^\Psi_{\Lambda}  := z\in\R \nonumber \\
 \mapsto & \;  \inf_{\alpha\in\Ac} \Big[\{\widehat{g}(\P_{X_T^{\alpha}})- Z^{z,\alpha}_{T}\}_+ + \Lambda  \int_0^T \{\Psi(s,\P_{X_s^{\alpha}})\}_+\ \di s + \Lambda  \{ \varphi( \P_{X_T^{\alpha}}) \}_+ \Big].  
\end{align}
\end{rem}

We discretize the problem in time and use a neural network by time step, since a single network taking time as input is usually not sufficient enough for complex problems, as shown in \cite{W21}. 
In view of the discussion about closed-loop controls in Section \ref{sec: integral}, 
the neural network representing the control  at each time step takes as inputs the current states  
$X$ and $Z^{z,\alpha}_{i}$ where $z$ is taken on a discretization of $K$. The method is described in Algorithm \ref{algo} with an example in Section \ref{sec: mean variance}. Solving \eqref{eq: Bellman} with the approach of \cite{GLPW21} would provide another numerical method for mean-field control with state constraints. 
The extension to the common noise case, where we aim to solve the  auxiliary problem  
\begin{align} \label{eq:algo2} 
& \;  \Yc^\Psi_{\Lambda}  := z\in\R \nonumber \\
 \mapsto & \;  \inf_{(\alpha,\beta)\in\Ac\times L^2(\F^0,\R^p)} \Big[\{\widehat{g}(\P^{W^0}_{X_T^{\alpha}})- Z^{z,\alpha}_{T}\}_+ + \Lambda  \int_0^T \{\Psi(s,\P^{W^0}_{X_s^{\alpha}})\}_+\ \di s + \Lambda  \{ \varphi( \P^{W^0}_{X_T^{\alpha}}) \}_+ \Big].  
\end{align} is given in Algorithm \ref{algo common} where the neural network for the control at each time step $t_i$ takes in addition as input the current value of the common noise $W_{t_i}^0$. Notice that in general, the control may depend on  the past values of the common noise, 
which could be taken into account in the neural network by taking as inputs the 
past increments of the common noise $\Delta W^0_0,\ldots,\Delta W^0_{i-1}$, where $\Delta W_i^0$ $=$ $W_{t_{i+1}}^0 -W_{t_i}^0$. 
The neural network for the auxiliary control $\beta$ at each time $t_i$ 
takes as inputs the current state $Z^{z,\alpha}_{i}$ and the current value of the common noise. 
An illustration is given in Section \ref{sec: storage}.
\\

\begin{algorithm2e}[H]
\DontPrintSemicolon 
\SetAlgoLined 

{\bf Input parameters:}  $\Lambda$, $M$, $N$, $N_T$, $\varepsilon$. 
{For a discretization $z_1<\cdots< z_M$ of $K$, minimize over neural networks $(\alpha_i)_{i\in 0,\cdots,N_T-1} $: $\R^d\times\R$ $\rightarrow$ $\R^q$ the loss function
\begin{align}
\sum_{m=1}^M w_{\Lambda}(z_m)
\end{align}
with $w_{\Lambda}$ defined by
\begin{align}
w_{\Lambda}(z)  & := \; \E \Big[\{\frac{1}{N}\sum_{l=1}^N g\Big(X_{N_T}^l,\frac{1}{N}\sum_{j=1}^N \delta_{X_{N_T}^j}\Big)- Z^{z,\alpha}_{N_T}\}_+ + \Lambda  \sum_{i=1}^{N_T} \{\Psi\Big(t_i,\frac{1}{N}\sum_{j=1}^N \delta_{X_{i}^j}\Big)\}_+\ \Delta t \\ & \ \ \ \ \ \ \ \ + \Lambda \{\varphi \Big(\frac{1}{N}\sum_{j=1}^N \delta_{X_{N_T}^j}\Big) \}_+ \Big].  
\end{align}\tcc*{ Auxiliary problem}
and for $i=0,\cdots,N_T-1$, $j=1,\cdots,N$
\begin{align}
    X_{i+1}^j & = X_{i}^j + b\big(t_i, X_{i}^j,\alpha_i(X_{i}^j,Z^{z,\alpha}_{i}),\overline{\mu}_i\big)\Delta t + \sigma\big(t_i, X_{i}^j,\alpha_i(X_{i}^j,Z^{z,\alpha}_{i}),\overline{\mu}_i\big)\Delta W_i^j, \quad X_{0}^j \sim  \mu_0 \\
    Z^{z,\alpha}_{i+1} &  =  Z^{z,\alpha}_{i} - \frac{1}{N} \sum_{l=1}^N f\big(t_i, X_{i}^l,\alpha_i (X_{i}^l,Z^{z,\alpha}_{i}), \overline{\mu}_i\big)\ \Delta t, \quad Z^{z,\alpha}_{0}  = z \\
    \overline{\mu}_i & = \frac{1}{N}\sum_{j=1}^N \delta_{(X_{i}^j,\alpha_i(X_{i}^j,Z^{z,\alpha}_{i}))}\\
\end{align}\tcc*{ Particle approximations}
Define $\alpha^*$ as the solution to this minimization problem.\\
Then, compute $V_0 = \inf \{z_i,\ i \in \llbracket 1, M \rrbracket \ |\ w_{\Lambda}(z_i) \leq \varepsilon \} $ with $\alpha = \alpha^*$ in the dynamics.\\\tcc*{ Recovering the cost of the original problem}
Return the value $V_0$ and the optimal controls $\hat{\alpha}_i : x \mapsto \alpha_i^*(x,Z^{V_0,\alpha^*}_i)$ for $i=0,\cdots,N_T-1$. 
\\ \tcc*{ Recovering the control of the original problem}
}
\caption{Algorithm to solve mean-field control problem \eqref{eq:algo1} \label{algo} }
\end{algorithm2e}

\begin{algorithm2e} 
\DontPrintSemicolon 
\SetAlgoLined 
{\bf Input parameters:}  $\Lambda$, $M$, $N$,  $N_T$, $\varepsilon$. 
{For a discretization $z_1<\cdots< z_M$ of $K$, minimize over neural networks $(\alpha_i)_{i\in 0,\cdots,N_T-1} $: $\R^d\times\R\times\R^p$ $\rightarrow$ $\R^q$ and $(\beta_i)_{i\in 0,\cdots,N_T-1} $: $\R\times\R^p$ $\rightarrow$ $\R^p$ the loss function 
\begin{align}
\sum_{m=1}^M w_{\Lambda}(z_m)
\end{align}
with $w_{\Lambda}$ defined by
\begin{align} 
w_{\Lambda}(z)  & := \; \E \Big[\{\frac{1}{N}\sum_{l=1}^N g\Big(X_{N_T}^l,\frac{1}{N}\sum_{j=1}^N \delta_{X_{N_T}^j}\Big)- Z^{z,\alpha,\beta}_{N_T}\}_+ + \Lambda \sum_{i =1}^{ N_T} \{\Psi\Big(t_i,\frac{1}{N}\sum_{j=1}^N \delta_{X_{i}^j}\Big)\}_+ \ \Delta t \\ & \ \ \ \ \ \ \ \ + \Lambda \{\varphi \Big(\frac{1}{N}\sum_{j=1}^N \delta_{X_{N_T}^j}\Big) \}_+ \Big]. 
\end{align}\tcc*{ Auxiliary problem}
and for $i=0,\cdots,N_T-1$, $j=1,\cdots,N$
\begin{align}
    X_{i+1}^j & = X_{i}^j + b\big(t_i, X_{i}^j,\alpha_i(X_{i}^j,Z^{z,\alpha,\beta}_{i},W_{t_{i}}^0),\overline{\mu}_i\big)\Delta t  + \sigma\big(t_i, X_{i}^j,\alpha_i(X_{i}^j,Z^{z,\alpha,\beta}_{i},W_{t_{i}}^0),\overline{\mu}_i\big)\Delta W_i^j\\ & + \sigma_0\big(t_i, X_{i}^j,\alpha_i(X_{i}^j,Z^{z,\alpha,\beta}_{i},W_{t_{i}}^0),\overline{\mu}_i\big)\Delta W_i^0, 
    \quad X_{0}^j   \sim  \mu_0 \\
    Z^{z,\alpha,\beta}_{i+1} &  = Z^{z,\alpha,\beta}_{i} - \frac{1}{N}  \sum_{l=1}^N f\big(t_i, X_{i}^l,\alpha_i (X_{i}^l,Z^{z,\alpha,\beta}_{i},W_{t_{i}}^0), \overline{\mu}_i\big)\ \Delta t  +  \beta_i(Z^{z,\alpha,\beta}_{i},W_{t_{i}}^0)\ \Delta W_{i}^0, \quad Z^{z,\alpha,\beta}_{0} & = z \\
    \overline{\mu}_i & = \frac{1}{N}\sum_{j=1}^N \delta_{(X_{i}^j,\alpha_i(X_{i}^j,Z^{z,\alpha,\beta}_{i},W_{t_{i}}^0))}\\
\end{align}\tcc*{ Particle approximations}
Define $(\alpha^*,\beta^*)$ as the solution to this minimization problem.\\
Then, compute $V_0 = \inf \{z_i,\ i \in \llbracket 1, M \rrbracket \ |\ w_{\Lambda}(z_i) \leq \varepsilon \} $ with $\alpha = \alpha^*$ and $\beta = \beta^*$ in the dynamics.\\\tcc*{ Recovering the cost of the original problem}
Return the value $V_0$ and the optimal controls $\hat{\alpha}_i : x \mapsto \alpha_i^*(x,Z^{V_0,\alpha^*,\beta^*}_i,W_{t_{i}}^0)$ for $i=0,\cdots,N_T-1$. 
\\ \tcc*{ Recovering the control of the original problem}
}
\caption{Algorithm to solve mean-field control problem \eqref{eq:algo2} \label{algo common} }
\end{algorithm2e}

\subsection{Mean-variance problem with state constraints}\label{sec: mean variance}

We consider the celebrated Markowitz portfolio selection problem where an investor can invest at any time $t$ an amount $\alpha_t$ in a risky asset (assumed for simplicity to follow a Black-Scholes model with constant rate of return $r$ and volatility $\sigma$ $>$ $0$), hence generating a wealth process $X$ $=$ $X^\alpha$ with dynamics
\begin{align}
\di X_t &= \; \alpha_t r\ \di t + \alpha_t \sigma\ \di W_t, \quad 0 \leq t \leq T, \; X_0 = x_0 \in \R.  
\end{align}
The goal is then to minimize over portfolio control $\alpha$ the mean-variance criterion : 
\begin{align}\label{eq: dual mean-variance}
\inf_\alpha\ J(\alpha) =\ & \lambda {\rm Var}(X_T^\alpha) - \E[X_T^\alpha]
\end{align} where $\lambda$ $>$ $0$ is a parameter related to the risk aversion of the investor.
We will add to this standard problem a conditional expectation constraint in the form 
\begin{align}
    & E[X_t^\alpha\ |\ X_t^\alpha\leq \theta] \geq \delta,\ \mathrm{if}\ \P(X_t^\alpha\leq \theta)> 0, 
\end{align} 
with $\delta < \theta$, which can be reformulated as
\begin{align}
    &  0\geq (\delta - E[X_t^\alpha\ |\ X_t^\alpha\leq \theta]) \P(X_t^\alpha\leq \theta).
\end{align} 
The auxiliary deterministic unconstrained control problem is therefore 
\begin{align}\label{eq: auxiliary problem mean variance}
   \Yc_{\Lambda}(z) := 
  \inf_{\alpha\in\Ac} \Big[\{\lambda {\rm Var}(X_T^{\alpha}) - \E[X_T^{\alpha}]- z\}_+ + \Lambda \int_0^T \{(\delta-E[X_s^{\alpha}\ |\ X_s^{\alpha}\leq \theta])\P (X_s^{\alpha}\leq\theta)\}_+\ \di s \Big]\\
\end{align} with the dynamics $\di X_s^{\alpha}= \alpha_s r\ \di s + \alpha_s \sigma\ \di W_s$, 
which corresponds to the constraint function $\Psi(t,\mu)\mapsto (\delta-E_\mu[\xi\ |\ \xi\leq \theta])\mu((-\infty,\theta]) $.  
We have the representation
$J(\alpha^*) = \Zc = \inf \{ z\in\R\ |\ \Yc_{\Lambda}(z) = 0 \}$. Indeed we see that the null control is admissible with the modified constraint $E[X_t^\alpha\ |\ X_t^\alpha\leq \theta]\P(X_t^\alpha\leq \theta)=0\geq (\delta+\varepsilon)\P(X_t^\alpha\leq \theta)=0,\ \forall t\in[0,T]$ for any $0<\varepsilon< \theta - \delta$ because $x_0\geq\theta$ hence $\P(X_t^\alpha\leq \theta) = 0$ so we can apply Theorem \ref{th: representation}.  
For practical application, other constraints could be considered like almost sure constraints on the portfolio weights as in \cite{W21b}. Instead of the dualization method used by \cite{LLP20}, constraints on the law of the tracking error with respect to a reference portfolio could be enforced.
\paragraph{}
For numerical tests we take $r=0.15,\ \sigma=0.35,\ \lambda=1$. We choose $x_0=1,\theta = 0.9,\ \delta =0.8$ and solve 
\begin{align}\label{eq: constraint mean variance}
\inf_\alpha\ & J(\alpha) = \lambda {\rm Var}(X_T^\alpha) - \E[X_T^\alpha]\\
    & \di X_t = \; \alpha_t r\ \di t + \alpha_t \sigma\ \di W_t,\\
    &  (0.8- E[X_t^\alpha\ |\ X_t^\alpha\leq 0.9] )\P(X_t^\alpha\leq 0.9) \leq 0,\ \forall t\in[0,T].
\end{align}We compare the controls from Algorithm \ref{algo} with the exact optimal ones in the unconstrained case for which we have an analytical value. We also solve without constraints for comparison and plot the final time histograms. We solve the unconstrained case with algorithm \ref{algo} and the one from \cite{CL19} for comparison. We take $50$ time steps for the time discretization and a batch size of 20000. We use an feedforward architecture with two hidden layers of 15 neurons. We perform 15000 gradient descent iterations with Adam optimizer (see \cite{KB14}) thanks to the Tensorflow library. The true value $v =J(\alpha^*)$ is -1.05041 without constraints. We also have the upper bound $-1.$ for the value in the constrained case, corresponding to the identically null control and wealth process $X_t = 1\ \forall t\in[0,T]$.  
With constraint we choose $K = [-1.047,-1.041]$, 
without constraint we take $K = [-1.07,-1.03]$, discretized by regular grids with 25 points.

\begin{figure}[H]
   \begin{minipage}[c]{.49\linewidth}
          \includegraphics[width=0.95\linewidth]{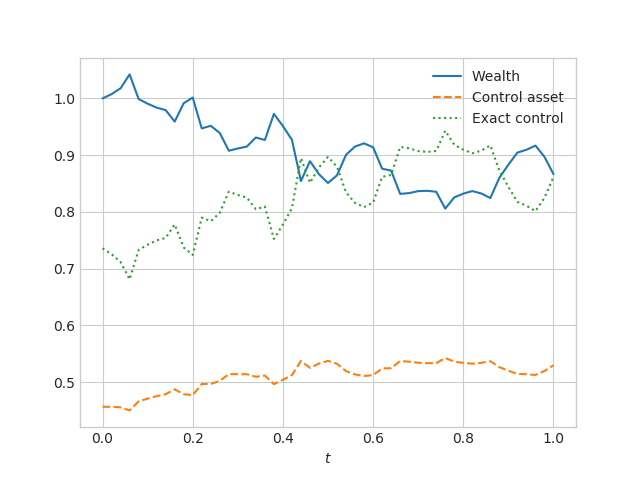}
          \caption*{Problem \eqref{eq: constraint mean variance}}
   \end{minipage} \hfill
   \begin{minipage}[c]{.49\linewidth}
      \includegraphics[width=0.95\linewidth]{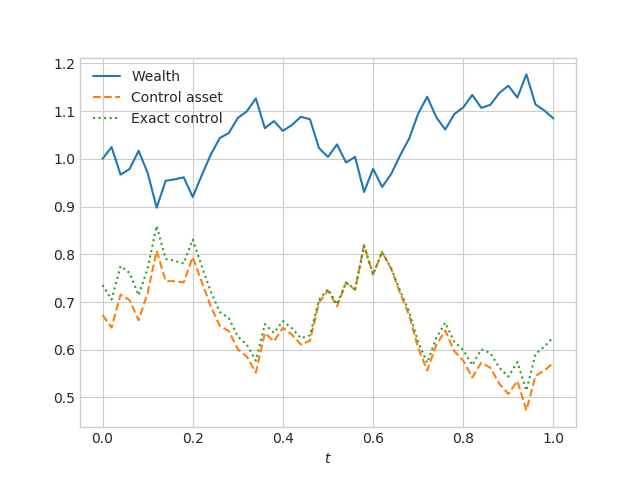}
      \caption*{No constraint, problem \eqref{eq: dual mean-variance}}
   \end{minipage}
   \caption{Sample path of the controlled process $X_t^\alpha$, with the analytical optimal control (for the unconstrained case) and the computed control. On the left figure we don't have the true control but plot the unconstrained one for comparison. Here $\Lambda = 100$} \label{fig: traject mean variance}
\end{figure}

\begin{figure}[H]
   \begin{minipage}[c]{.49\linewidth}
          \includegraphics[width=0.95\linewidth]{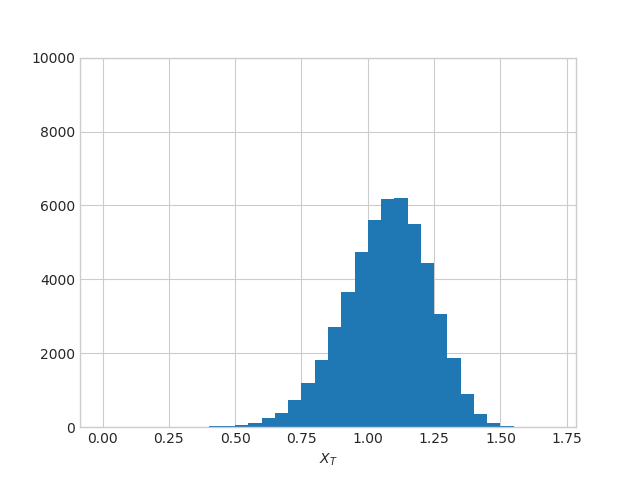}
          \caption*{Problem \eqref{eq: constraint mean variance}}
   \end{minipage} \hfill
   \begin{minipage}[c]{.49\linewidth}
      \includegraphics[width=0.95\linewidth]{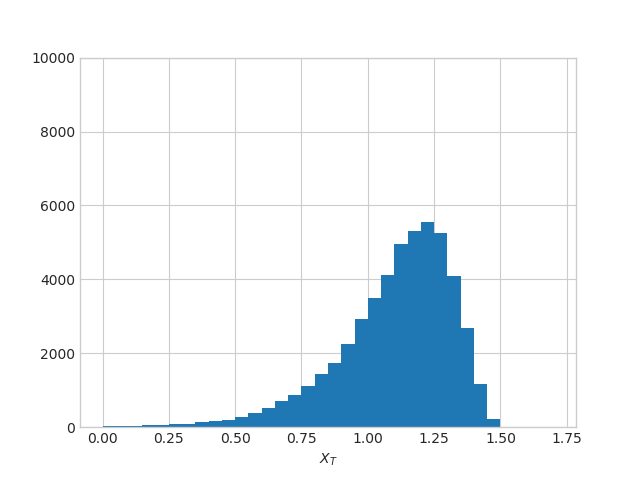}
      \caption*{No constraint, problem \eqref{eq: dual mean-variance} }
   \end{minipage}
   \caption{Histogram of $X_T^\alpha$ for 50000 samples. Here $\Lambda = 100$.} \label{fig: hist mean variance}
\end{figure}

\begin{figure}[H]
   \begin{minipage}[c]{.49\linewidth}
          \includegraphics[width=0.95\linewidth]{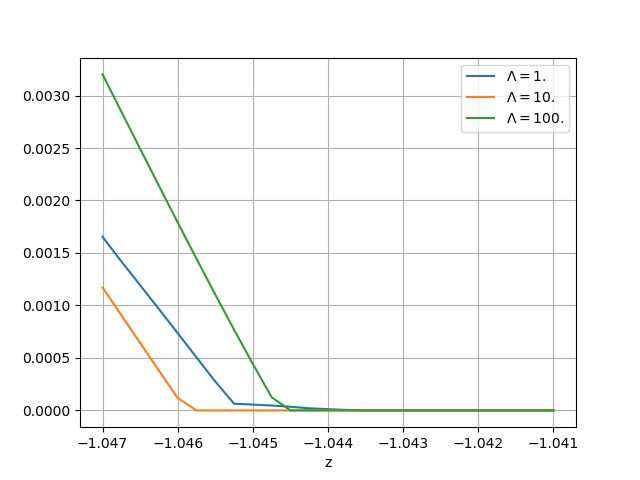}
          \caption*{Problem \eqref{eq: constraint mean variance}}
   \end{minipage} \hfill
   \begin{minipage}[c]{.49\linewidth}
      \includegraphics[width=0.95\linewidth]{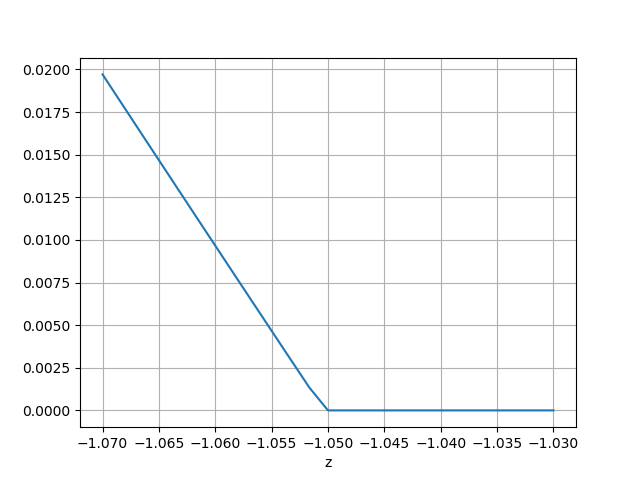}
      \caption*{No constraint, problem \eqref{eq: dual mean-variance} }
   \end{minipage}
   \caption{Auxiliary value function $\Yc_{\Lambda}(z)$ for several values of $\Lambda$ in the constrained case, auxiliary value function $\Yc(z)$ in the unconstrained case} \label{fig: w values}
\end{figure} 

\begin{figure}[H]
   \begin{minipage}[c]{.49\linewidth}
          \includegraphics[width=0.95\linewidth]{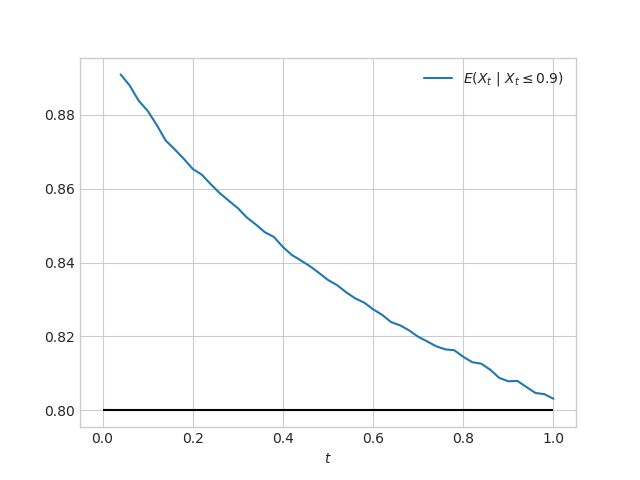}
          \caption*{Problem \eqref{eq: constraint mean variance}}
   \end{minipage} \hfill
   \begin{minipage}[c]{.49\linewidth}
      \includegraphics[width=0.95\linewidth]{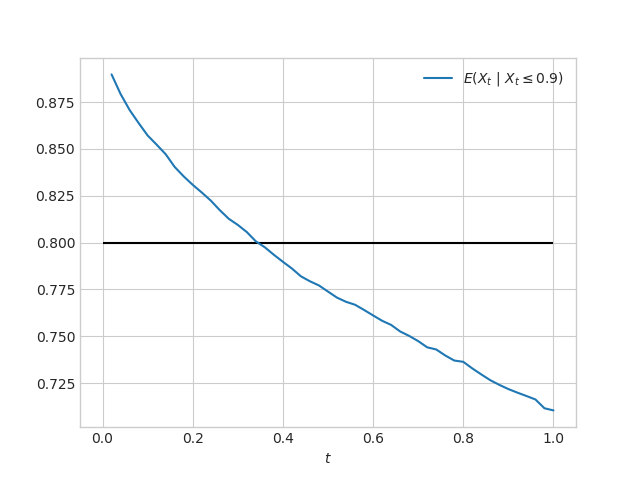}
      \caption*{No constraint, problem \eqref{eq: dual mean-variance} }
   \end{minipage}
   \caption{Conditional expectation $E[X_t^\alpha\ |\ X_t^\alpha\leq 0.9]$ estimated with 50000 samples. The black line corresponds to $\delta = 0.8$. Here $\Lambda = 100$} \label{fig: proba}
\end{figure}

In Figure \ref{fig: hist mean variance} we observe the shift of the distribution of the final wealth thanks to the constraint (on the left) with less probable large losses but also less probable large gains. Indeed Figure \ref{fig: proba} confirms that the conditional expectation constraint is verified when we solve the corresponding problem through our level set approach. We see in Figure \ref{fig: w values} that the more $\Lambda$ is large the more the auxiliary value function becomes affine before reaching zero. Additional results are presented in Table \ref{tab: table results mean variance}. 

Our method can also handle directly the primal of the mean-variance problem, that is to maximize over portfolio control $\alpha$ the expected terminal wealth under a terminal variance constraint: 
\begin{align}\label{eq: primal mean-variance}
\inf_\alpha\ & \bar J(\alpha) =   - \E[X_T^\alpha]\\
    & \di X_t = \; \alpha_t r\ \di t + \alpha_t \sigma\ \di W_t,\\
    & {\rm Var}(X_T^\alpha ) \leq \vartheta.
\end{align} which give the same optimal control as Problem \eqref{eq: dual mean-variance} under the correspondence $\lambda =  \sqrt{\frac{\exp(\sigma^{-2}r^2 T)-1}{4\vartheta}}$. This problem allows us to consider a constrained problem with an analytical solution.
In this case ${\rm Var}(X_T^{\alpha^*} )= \vartheta$ thus 
$J(\alpha^*) = \lambda {\rm Var}(X_T^{\alpha^*} ) + \bar J(\alpha^*)= \lambda \vartheta + \bar J(\alpha^*)$. For comparison with Problem \eqref{eq: dual mean-variance}  we thus report $\lambda \vartheta + \bar J(\alpha^*)$ for Problem \eqref{eq: primal mean-variance} in Table \ref{tab: table results mean variance} and choose $\vartheta = \frac{\exp(\sigma^{-2}r^2 T)-1}{4\lambda^2} = 0.0504$. In this case the auxiliary deterministic unconstrained control problem is now
\begin{align}\label{eq: auxiliary problem mean variance primal}
   \Uc_{\Lambda}(z) = \inf_{\alpha\in\Ac} &\Big[\{ - \E[X_T^{\alpha}]- z\}_+ + \Lambda \{ \mathrm{Var}(X_T^{\alpha})-\vartheta\}_+ \Big]    \\
  & \di X_t = \; \alpha_t r\ \di t + \alpha_t \sigma\ \di W_t,
\end{align} which corresponds to the constraint function $\Psi(t,\mu)\mapsto ( \mathrm{Var}(\mu)-\vartheta)_+\mathds{1}_{t=T}$ and the modified constraint function $\overline{\Psi}_\eta(t,\mu)\mapsto ( \mathrm{Var}(\mu)-\vartheta)_+\mathds{1}_{t=T}-\eta \mathds{1}_{t<T}$ (see Remark \ref{rem: finite rhs}). 
Theorem \ref{th: representation} still applies as far as the null control is admissible with the modified constraint $(\mathrm{Var}(\mu)-\vartheta)_+\mathds{1}_{t=T} +\varepsilon - \eta \mathds{1}_{t<T}\leq 0$ for any $0 < \varepsilon < \eta$ and any $t\in[0,T]$.  

\begin{figure}[H]
   \begin{minipage}[c]{.49\linewidth}
          \includegraphics[width=0.95\linewidth]{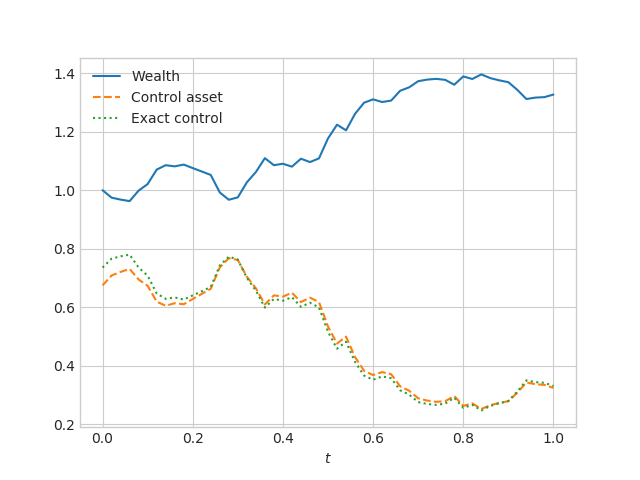}
   \end{minipage} \hfill
   \begin{minipage}[c]{.49\linewidth}
      \includegraphics[width=0.95\linewidth]{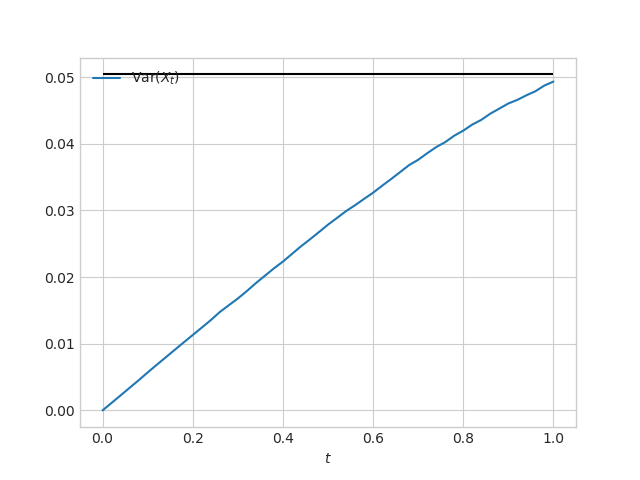}
   \end{minipage}
   \caption{Sample trajectory of the controlled process $X^\alpha_t$ and the control for problem \eqref{eq: primal mean-variance} (left). Variance $\mathrm{Var}(X_t)$ estimated with 50000 samples for problem \eqref{eq: primal mean-variance} (right) with $\Lambda = 10$} \label{fig: traject mean variance primal}
\end{figure}

\begin{figure}[H]
    \centering
    \includegraphics[width=8cm]{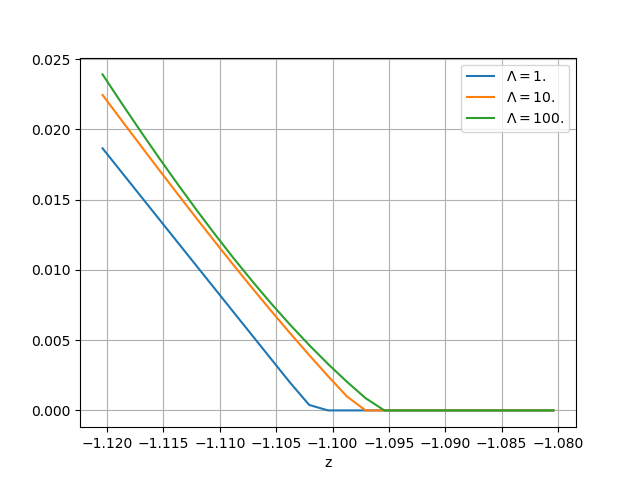}
    \caption{Auxiliary value function $\Uc_{\Lambda}(z)$ for several values of $\Lambda$}
    \label{fig: w primal}
\end{figure}

Figure \ref{fig: traject mean variance primal}  
shows that we recover the optimal control for the problem and that the terminal variance constraint is satisfied. We see in Figure \ref{fig: w primal} that similarly as in Figure \ref{fig: w values}, for large values of $\Lambda$ the auxiliary value function is affine before reaching zero. In this case the exact solution is $-1.10$ which is very close to the point in which the affine part reaches zero.



\begin{scriptsize}
\begin{table}[htp]
	\centering
	\begin{tabular}{|c|c|c|c|c|c|c|c|c|c|}
		\hline
		Problem & Average & Std & Tr. val. & Error & $\E[X_T^{\alpha^*}]$ & Tr. $\E[X_T^{\alpha^*}]$ & $\mathrm{Var}(X_T^{\alpha^*})$ & Tr. $\mathrm{Var}(X_T^{\alpha^*})$\\	
		 \hline
		 \eqref{eq: constraint mean variance} $\Lambda = 1.$  &   -1.044   & 0.0010 & {\rm Not avail.} & {\rm Not avail.} & 1.07 & {\rm Not avail.} & 0.026  & {\rm Not avail.} \\
		\hline
	  \eqref{eq: constraint mean variance} $\Lambda = 10.$  & -1.044   & 0.0005 & {\rm Not avail.} & {\rm Not avail.} & 1.07 & {\rm Not avail.} &  0.026 & {\rm Not avail.} \\
		\hline
		 

	 \ref{eq: constraint mean variance}	$\Lambda = 100.$ & -1.045
 & 0.0005 & {\rm Not avail.}  & {\rm Not avail.}  & 1.07  & {\rm Not avail.}  & 0.027 & {\rm Not avail.}  \\
	\hline
	

	 \eqref{eq: primal mean-variance} $\Lambda = 10.$ & -1.048
 & 0.0017 & -1.050 & 0.22 & 1.10 & 1.10 & 0.049 & 0.050\\
	\hline
		\eqref{eq: dual mean-variance}  & -1.050
   & 0.0009  & -1.050 & 0.07 &  1.10 & 1.10 & 0.050 & 0.050\\
		\hline
		\eqref{eq: dual mean-variance} \cite{CL19} & -1.052   & 0.0022 & -1.050 & 0.13  & 1.10 & 1.10 & 0.053  & 0.050\\
		\hline
	\end{tabular}
	\caption{Estimate of the solution 
	with maturity $T = 1.$ 
	Average and standard deviation observed over 10 independent runs are reported, with the relative error (in $\%$). We also report the terminal expectation and variance of the approximated optimally controlled process for a single run. '{\rm Not avail.} ' means that we don't have a reference value and 'Tr.' means true. For problem \eqref{eq: primal mean-variance},  we take $\Lambda = 10$ and for problem \eqref{eq: constraint mean variance} we illustrate the values obtained for $\Lambda \in \{1., 10., 100\}$. }
	\label{tab: table results mean variance}
\end{table}
\end{scriptsize}

In Table \ref{tab: table results mean variance} we observe that our method gives a small variance for the results over several runs. In the case where an analytical solution is known, the value of the control problem is computed accurately with less than $0.5\%$ of relative error. The expectation and variance of the terminal value of the optimally controlled process are also very close to their theoretical values. In the case of a conditional expectation constraint, even though we don't have an exact solution we notice that the value is close to the unconstrained value hence since our solution is admissible, we expect to be near optimality. On the unconstrained problem \eqref{eq: dual mean-variance} our scheme and the one from \cite{CL19} give similar results.
 
\subsection{Optimal storage of wind-generated electricity}\label{sec: storage}

We consider $N$ wind turbines with $N$ associated batteries. Define the productions $P^i_t$, storage levels $X^i_t$, storage injection $\alpha^i_t$ for which we provide a typical range\footnote{https://css.umich.edu/factsheets/us-grid-energy-storage-factsheet}. We consider the following constraints
\begin{align}
    \begin{cases}
        & 0 \leq X^i_t \leq X_{\max}\longrightarrow \mathrm{limited}\ \mathrm{storage}\ \mathrm{capacity} \ (1 \mathrm{kWh} - 10 \ \mathrm{ MWh})\\
        & \underline{\alpha} \leq \alpha^i_t \leq \overline{\alpha} \longrightarrow \mathrm{limited}\ \mathrm{injection/withdrawal}\ \mathrm{capacity} \ (10\ \mathrm{ kW} - 10 \mathrm{ MW})
    \end{cases}
\end{align} with $X_{\max}\geq 0 $, $\underline{\alpha} \leq 0 \leq \overline{\alpha} $. 
Define the spot price of electricity $S_t$ without wind power, $\widetilde{S_t}$ the price with wind production. Selling a quantity $P^i_t - \alpha^i_t$ on the market, producer $i$ obtains a revenue $\widetilde{S_t}(P^i_t - \alpha^i_t)$ (if $P^i_t - \alpha^i_t < 0$ the producer is buying from the market) 
where the market price is affected by linear price impact 
\begin{equation}
    \widetilde{S_t} = S_t - \frac{\Theta(N)}{N} \sum_{i=1}^N (P^i_t - \alpha^i_t), 
\end{equation}
modeling the impact of intermittent renewable production on the market. $\Theta$ is positive, non-decreasing and bounded. We call $\Theta_\infty = \lim_{N\rightarrow \infty} \Theta(N) < \infty$. 
We consider $N + 2$ independent Brownian motions $W_t^0,B_t^0,W_t^1,\cdots,W_t^N$  and the following dynamics for the producers $i=1,\cdots,N$ state processes
\begin{align}
    \begin{cases}
        & \di X^i_t = \alpha^i_t\ \di t\\
        & \di P^i_t = \iota (\phi P_{\max} - P_t^i)\ \di t +  \sigma_p(P_t\wedge \{P_{\max} - P_t^i\} )_+(\rho\ \di W_t^0 + \sqrt{1-\rho^2}\ \di W_t^{i})\\
        & \di F(t,T) = F(t,T) \sigma_f  e^{-a(T-t) }\di B_t^0 \nonumber\\
        & S_t = F(t,t) \nonumber,
    \end{cases}
\end{align}
for some positive constants  $\kappa$, $\phi$, $P_{\max}$, $\sigma_p$ $\sigma_f$, and  $\rho$ $\in$ $[-1,1]$. In the production dynamics, the common noises $W_t^0,B_t^0$ corresponds to the global weather and the market price randomness whereas the idiosyncratic noises $W_t^i$ for $i>1$  model the local weather, independent from one wind turbine to another. We call $\F^0$ the filtration generated by $W_t^0,B_t^0$. The productions $ P^i_t $ are bounded processes and the price $S_t$ is positive. Of course the modified price $\widetilde{S_t}$ in the presence of renewable producers can become negative, as empirically observed in some overproduction events. However it stays bounded by below in our model. Producer $i$ gain function to maximize is
\begin{align}
    J^i(\alpha_1,\cdots,\alpha_N) & = \E\Big[\int_0^T \{S_t(P^i_t - \alpha^i_t) - \frac{\Theta(N)}{N} (P^i_t - \alpha^i_t)\sum_{j=1}^N (P^j_t - \alpha^j_t) \}\ \di t \Big].  
\end{align}
The related mean field control problem for a central planner is therefore 
\begin{align}
    - \inf_{\alpha\in\Ac}\ &  \E\Big[\int_0^T \{-S_t(P_t - \alpha_t) + \Theta_\infty (P_t - \alpha_t) \E[P_t - \alpha_t | \F^0] \}\ \di t \Big]\\
        & \di X_t = \alpha_t\ \di t\\ 
          & \di P_t = \iota (\phi P_{\max} - P_t)\ \di t  + \sigma_p(P_t\wedge \{P_{\max} - P_t\} )_+(\rho\ \di W_t^0 + \sqrt{1-\rho^2}\ \di W_t^1)\\
        & \di F(t,T) = F(t,T) \sigma_f  e^{-a(T-t) }\di B_t^0 \nonumber\\
        & S_t = F(t,t) \nonumber\\ 
        & 0 \leq X_t \leq X_{\max}\ \P\ \mathrm{a.s.}
 \end{align}
Here the state is  $ (X_t, P_t, S_t)\in \R^3$ hence the distribution of the state lives in $\Pc_2(\R^3)$. The set $\Ac$ corresponds to progressively measurable controls with values in the compact set  $[\underline{\alpha} , \overline{\alpha}]$.
A similar problem is solved by \cite{ABM20} without any storage constraints by Pontryagin principle. With constraints but without mean-field interaction, a close problem is solved by \cite{PV19}. 
For instance $X_\mathrm{max} = 0$ corresponds to the much simpler problem without storage nor control of the valuation of a wind power park. See also \cite{CK21,W21}. To represent the almost sure constraint $0 \leq X_t \leq X_{\max}$ we choose as constrained function
\begin{equation}
    \Psi : \mu\in\Pc_2(\R^3) \mapsto \int_{\R}  \{(-x)_+^2 + (x-X_\mathrm{max})_+^2\}\ \mu_1(\di x),
\end{equation} where $\mu_1$ is the first marginal law of the measure $\mu$.

The auxiliary unconstrained control problem is therefore
\begin{align}\label{eq: auxiliary problem storage}
   w(z)   := -\inf_{\alpha,\beta^{0,1},\beta^{0,2}\in\Ac\times L^2\times L^2} \E \Big[& \{\int_0^T \E[-S_t(P_t - \alpha_t) + \Theta_\infty (P_t - \alpha_t) \E[P_t - \alpha_t | \F^0] | \F^0] \ \di t - z \\ & - \int_0^T \beta_s^{0,1}\ \di W_s^0- \int_0^T \beta_s^{0,2}\ \di B_s^0\}_+ + \frac{1}{\epsilon}\int_0^T \E[(-X_s)_+^2 + (X_s-X_{\mathrm{max}})_+^2]\ \di s \Big]   \\
        & \di X_t = \alpha_t\ \di t\\
        & \di P_t = \iota (\phi P_{\max} - P_t)\ \di t + \sigma_p(P_t\wedge \{P_{\max} - P_t\} )_+\ (\rho\ \di W_t^0 + \sqrt{1-\rho^2}\ \di W_t^1)\\
        & \di F(t,T) = F(t,T) \sigma_f  e^{-a(T-t) }\di B_t^0 \nonumber\\
        & S_t = F(t,t) \nonumber 
\end{align}
where $\epsilon$ is a small term used to force the a.s. constraints.\\

We  consider the standard stochastic control benchmark with only common noise for the production $(\rho=1)$.  
It corresponds to a single very large wind farm where all wind turbines produce the same power. 
The problem degenerates as
\begin{align}
     \label{eq:EoleDeg}
    - \inf_{\alpha\in\Ac}\ &  \E\Big[\int_0^T \{( -S_t +  \Theta_\infty (P_t - \alpha_t))  (P - \alpha_t)  \}\ \di t \Big] \nonumber\\
        & \di X_t = \alpha_t\ \di t \nonumber\\
        & \di P_t = \iota (\phi P_{\max} - P_t)\ \di t + \sigma_p(P_t\wedge \{P_{\max} - P_t\} )_+\  \di W_t^0  \nonumber\\
        & \di F(t,T) = F(t,T) \sigma_f  e^{-a(T-t) }\di B_t^0 \nonumber\\
        & S_t = F(t,t) \nonumber\\
        & 0 \leq X_t \leq X_{\max}\ \P\ \mathrm{a.s.}
 \end{align} 
 and  equation \eqref{eq: auxiliary problem storage}
 gives 
 \begin{align}\label{eq: auxiliary problem storage deg}
   w(z)   := -\inf_{\alpha,\beta\in\Ac\times L^2} \E \Big[& ((Y^{\alpha,\beta}-z)_+  +   \frac{1}{\epsilon}\int_0^T\E[(-X_s)_+^2 + (X_s-X_{\mathrm{max}})_+^2]\ \di s \Big]   \\
        & \di X_t = \alpha_t\ \di t\\
        & \di P_t = \iota (\phi P_{\max} - P_t)\ \di t + \sigma_p(P_t\wedge \{P_{\max} - P_t\} )_+\  \di W_t^0 \\
        & \di F(t,T) = F(t,T) \sigma_f  e^{-a(T-t) }\di B_t^0 \nonumber\\
        & S_t = F(t,t) \nonumber 
\end{align}
where
\begin{align}
\label{eq:Y}
    Y^{\alpha,\beta} = \int_0^T (-S_t+  \Theta_\infty (P_t - \alpha_t))(P_t - \alpha_t)  \ \di t  - \int_0^T \beta_s^{0,1}\ \di W_s^0- \int_0^T \beta_s^{0,2}\ \di B_s^0. 
\end{align}
The solution of our  optimization problem is then $z^* = \sup \{ z\ |\ \hat w(z) =0 \}$ where $\hat w(z) := -w(z)$.
 Remark now that \eqref{eq:Y} will be estimated discretizing  the integral  $\int_0^T \beta_s^{0,1}\ \di W_s^0$ and $\int_0^T \beta_s^{0,2}\ \di B_s^0$ using an Euler scheme for the underlying processes and therefore $\hat w(z)$ will be above $0$ except for low values of $z$ due to the variance of the $Y^{\alpha*,\beta}$ estimator that cannot be reduced to $0$.\\
 In order to reduce the variance of $Y^{\alpha, \beta}$, we propose to modify   $Y^{\alpha, \beta}$ as follows :
 \begin{align}
\label{eq:newY}
    Y^{\alpha,\beta} = & \int_0^T (-S_t+  \Theta_\infty (P_t - \alpha_t))(P_t - \alpha_t)  \ \di t -  \int_0^T (-S_t+  \Theta_\infty (P_t - \hat \alpha_t))(P_t - \hat \alpha_t)  \ \di t  + \\
    & \E[ \int_0^T (-S_t+  \Theta_\infty (P - \hat \alpha_t))(P_t - \hat \alpha_t)  \ \di t] - \int_0^T \beta_s^{0,1}\ \di W_s^0- \int_0^T \beta_s^{0,2}\ \di B_s^0
\end{align}
where $\hat \alpha_t$ is the rough estimation of the optimal {\bf deterministic} command maximizing the gain.\\
 We take $T = 40$, $N_T = 40$ time steps, $X_{\max} = 1$, $X_0=0.5$, $P_0 = 0.12$, $F(0,t) = 30  + 5\cos(\frac{2\pi t}{N}) + \cos(\frac{2\pi t}{7})$, $\sigma_f = 0.3$, $a=0.16$, $\iota = 0.2$, $\sigma_p = 0.2$, $\phi=0.3$, $P_{\max} = 0.2$, $-0.2\leq \alpha \leq 0.2$, $\Theta(N) = 10$. \\%
The network depends on $P_t, S_t,X_t$ and $z$ where $z$ takes some deterministic values on a grid with the same spacing.
The global curve is therefore approximated by a single run.\\
The grid is taken from $107$ to $127$ with a spacing  of $0.5$.
The neural networks have two hidden layers with 14 neurons on each layer.
We take a $\epsilon$ parameter equal to $10^{-4}$.
The number of gradient iterations is set to $50000$ with a learning rate equal to $ 2 \times 10^{-3}$ 
We give the $\hat w$ function on figure \ref{fig:LevelSetWStor}.
\begin{figure}[H]
    \centering
    \includegraphics[width=0.4\linewidth]{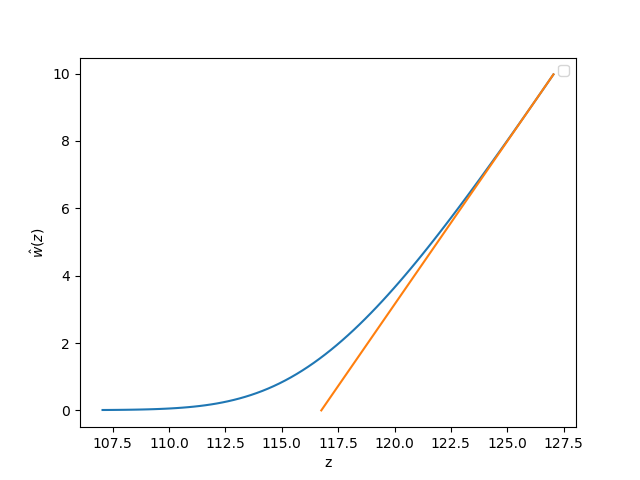}
    \caption{$\hat w$ function value for the storage problem}
    \label{fig:LevelSetWStor}
\end{figure}
Using  Dynamic Programming with the StOpt library \cite{gevret2018stochastic}, we get an optimal value equal to $117.28$ while  a direct optimization of \eqref{eq:EoleDeg} using some neural networks as in \cite{W21}, \cite{CL19} we get a value of $117.11$. Encouraged by Remark \ref{rem: lambda}, Figure \ref{fig: w values}, Figure \ref{fig: w primal} and the related comments, we empirically estimate the value function by the point where the linear part of the auxiliary function reaches zero when $\Lambda = \frac{1}{\varepsilon}$ is sufficiently large. The estimated value is $116.75$, close to the reference solutions.
On figure \ref{fig:trajStorage}, we compare trajectories obtained by Dynamic Programming and by the Level Set approach : they are accurately calculated.
\begin{figure}[H]
    \centering
    \begin{minipage}[c]{.32\linewidth}
    \includegraphics[width=\linewidth]{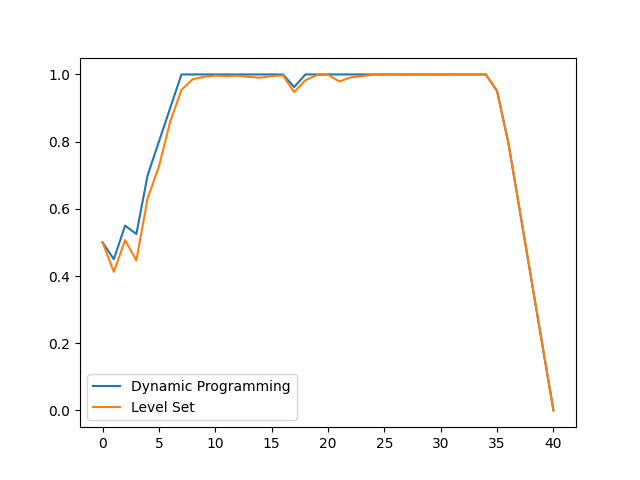}
    \end{minipage}
    \begin{minipage}[c]{.32\linewidth}
    \includegraphics[width=\linewidth]{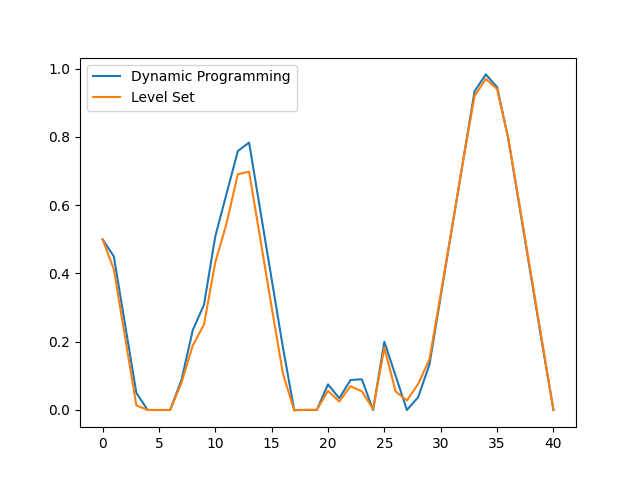}
    \end{minipage}
    \begin{minipage}[c]{.32\linewidth}
    \includegraphics[width=\linewidth]{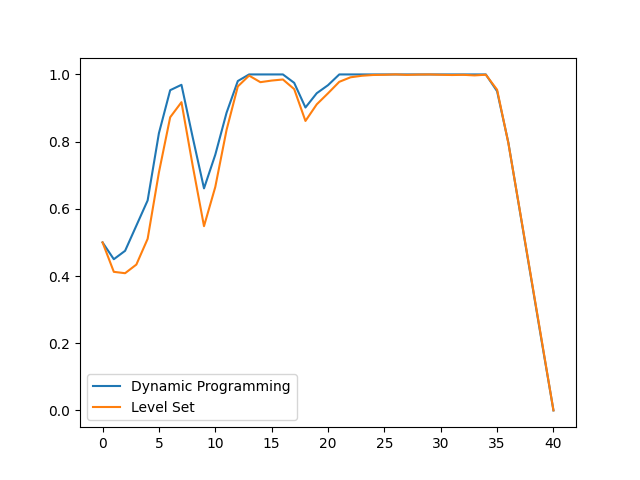}
    \end{minipage}
    \caption{Storage trajectories with the Level Set and Dynamic  Programming method.}
    \label{fig:trajStorage}
\end{figure}

\printbibliography

\end{document}